\pdfoutput=1

\documentclass[11pt]{article}

\usepackage[margin=1.1in]{geometry}

\usepackage{amsmath,amscd}
\usepackage{amssymb}
\usepackage{amsthm}
\usepackage{comment}
\usepackage{graphicx}
\usepackage{epstopdf} 
\usepackage{mathrsfs}
\usepackage{cite}
\usepackage[ocgcolorlinks, linkcolor=blue]{hyperref}
\usepackage{authblk}
\usepackage{enumerate}

\usepackage{bm}
\usepackage{hyperref}
\usepackage{xcolor,color,soul}
\hypersetup{
	colorlinks,
	linkcolor={blue},
	urlcolor={blue},
	citecolor={red}
}

\newtheorem{thm}{Theorem}[section]

\newtheorem{lem}[thm]{Lemma}
\newtheorem{cor}[thm]{Corollary}

\newtheorem{rmk}[thm]{Remark}

\newtheorem*{proposition*}{Proposition}

\numberwithin{equation}{section}

\newcommand{\norm}[1]{\lVert #1 \rVert}

\title{On inverse problems in multi-population aggregation models}

\author[1,*]{Yuhan Li}
\author[1,$\dagger$]{Hongyu Liu}
\author[1,$\natural$]{Catharine W. K. Lo}
\affil[1]{Department of Mathematics, City University of Hong Kong}
\affil[*]{yuhli2-c@my.cityu.edu.hk}
\affil[$\dagger$]{hongyu.liuip@gmail.com, hongyliu@cityu.edu.hk}
\affil[$\natural$]{wingkclo@cityu.edu.hk}
\date{}

\begin{document}
\maketitle

\begin{abstract}
This paper focuses on inverse problems arising in studying multi-population aggregations. The goal is to reconstruct the diffusion coefficient, advection coefficient, and interaction kernels of the aggregation system, which characterize the dynamics of different populations. In the theoretical analysis of the physical setup, it is crucial to ensure non-negativity of solutions. To address this, we employ the high-order variation method and introduce modifications to the systems. Additionally, we propose a novel approach called transformative asymptotic technique that enables the recovery of the diffusion coefficient preceding the Laplace operator, presenting a pioneering method for this type of problems. Through these techniques, we offer comprehensive insights into the unique identifiability aspect of inverse problems associated with multi-population aggregation models. 
~\\\\
\textbf{Keywords:} Inverse multi-population aggregation model; positive solutions;  unique identifiability; transformative asymptotic technique; high-order variation method.~\\
\textbf{2020 Mathematics Subject Classification:} 35R30, 35B09, 35K45, 35Q92, 92-10, 92D25, 92D50, 35B10, 35C20
\end{abstract}

{\centering \section{INTRODUCTION}  \label{intro} }
\subsection{Problem setup and background}\label{sec1_1}
In the natural world, we often use collective nouns to depict the characteristics of groups formed by different species. For instance, when referring to hyenas, we use the term `a clan of hyenas'. Similarly, we describe sharks as `a school of sharks'. Likewise, we use `a herd of cattle' instead of `a swarm of cattle' and `a sloth of bears' rather than `a flock of bears'. It is interesting to note that all of these animals mentioned above exhibit gregarious behavior, where the ability to aggregate is both important and habitual. This raises intriguing questions: Do different social animals share common traits? Why do we employ various quantifiers to describe them (cf.  \cite{cantrell2004, eftimie2018, edelstein2005, maini2001, jin2022global})? Interestingly, this phenomenon extends to the microscopic level as well; see e.g \cite{parrish1999,deisboeck2001,chowdhury2005} and the references cited therein. A cell may pass information to another through direct physical contact of the specialised molecules on its surface, or modify its motility based on the molecular signals released by another cell.

Organisms, whether animals, plants, or microorganisms, exhibit two distinct types of behaviors related to aggregative and movement: aggregation and movement in response to external environmental factors such as chemicals, and aggregation and movement resulting from interactions with individuals of the same species (self-organizing behaviors) \cite{eftimie2018}.  This study specifically focuses on analyzing self-organizing behaviors, which encompass a fascinating array of processes wherein complex patterns and structures arise from the interactions among individual components or agents. In mathematical biology, this refers to the spontaneous formation of spatial patterns, collective behavior, or functional organization without the need for external control or central coordination \cite{townes1955, steinberg, sumpter2010}. Instead, these patterns and behaviors arise from local interactions and feedback mechanisms among the system's components. This concept mirrors the self-organizing behaviors observed in nature, where intricate structures and dynamics emerge from the interactions of simple entities.

Up to now, there are several mathematical models proposed to study this self-organizing phenomena, providing a means to describe and analyze the complex dynamics of biological systems, capturing the interplay between various factors such as diffusion, advection, reaction kinetics, and feedback loops. These models often take the form of partial differential equations, stochastic processes, or agent-based simulations. Moreover, since repulsion act over shot distances while attraction act over large spatial distances, many mathematical models depicting self-organizing biological aggregations are non-local \cite{eftimie2018}.

In the study of self-organizing phenomena, non-local advection-diffusion equations have, in particular, proven to be a valuable tool for gaining insights and understanding. These equations incorporate both advection, representing the movement of substances or agents, and diffusion, accounting for their spreading or dispersal. The inclusion of non-local terms allows for the modeling of long-range interactions and memory effects, capturing the influence of distant points on the current state of the system. One classic advection-diffusion equation is of the following form:
\begin{equation}\label{class_eq1}
    \partial_t u(x,t)= \nabla \cdot[D\nabla u(x,t)-\mathbf{a} u(x,t)],
\end{equation}
where $u(x,t)$ represents the density of some population at position $x\in\Omega\subset\mathbb{R}^n$ and $t\in [0,\infty).$ $D$ is the diffusive component of movement, and $\mathbf{a}$ is an $n-$dimensional vector measuring the advective component of movement. The region $\Omega$ defines the region in which the population species moves.

The classic model \eqref{class_eq1} serves as a foundation for deriving different models that describe various scenarios. In the case of cellular systems, it is necessary to construct models which incorporate adhesion, which is the fundamental mechanism by which cells attach to and interact with their surrounding environment.  These models often originate from a random walk description of movement \cite{anguige2009,johnston2012} or are proposed based on empirical observations \cite{hofer1995}. In contrast, animals exhibit social interactions that lead to aggregation. At the most basic level, social interactions among animals can influence their spatial distribution, causing them to concentrate in specific areas rather than dispersing throughout the entire available space \cite{topaz2006, ballerini2008,cantrell2020}. For example, penguins exhibit social interactions within their colonies, forming dense aggregations in particular regions. These colonies provide protection, warmth, and opportunities for mating and raising offspring. The social interactions among penguins result in a concentrated distribution within the colony, rather than dispersing across the entire available habitat \cite{wilson2009, ancel2015}. By forming specific areas of concentration, animals can benefit from cooperative behaviors, enhanced protection, and efficient resource utilization within the social group. 

In this context, a question arises: as the specific area is fixed and various species lived there, an interaction range is generated, enabling individuals to sense multiple neighbours at the same time. Therefore, it is reasonable to infer that the movement of the individual will rely on an integrated response, specifically aligned with the distribution of the population (or populations) within its range of interaction. Non-local PDE formulations play a significant role in studying such phenomena and we consider the following pair of models to reflect such behavior:
\begin{equation}\label{convo1a}
    \partial_{t}u=d\Delta u -\mu \nabla \cdot[u \mathbf{k}_R \ast f], \quad \mathbf{k}_R \ast f(x,t)=\int_{\mathbb{T}^{n}} \mathbf{k}_R(x,y) f(u(y,t)) dy,
\end{equation}
\begin{equation}\label{convo1b}
    \partial_{t}u=d\Delta u -\nu \nabla \cdot[u\nabla (w_R \ast g)], \quad w_R \ast g(x,t)=\int_{\mathbb{T}^{n}} w_R(x,y) g(u(y,t)) dy.
\end{equation}

For the first model \eqref{convo1a}, the non-local advection term is based on the principle that the population at location $y$ affects the population movement at location $x.$ The integral kernel $\mathbf{k}_R $ is parametrised according to a sampling radius $R,$ which represents the interaction range. Here, $R$ is a vector which not only specifies the distance from $y$ to $x,$ but also determines the direction of interaction. $f(x,t)$ depicts the dependence on the population size at $y.$ The physical meaning of \eqref{convo1a} can be interpreted in the sense of forces and energies. From a cellular perspective, \eqref{convo1a} arises from the delicate interplay between adhesion and repulsion forces acting on the cell surface. The interaction between cells, centered at $x$ and $y$, gives rise to localized forces. The net force experienced by the cell is determined by the convolution $\mathbf{k}_R \ast f$. On the other hand, in the context of animal studies, \eqref{convo1a} has been developed to describe observed swarming-like behavior \cite{kawasaki1978, mimura1982, nagai1983}.

For the second model \eqref{convo1b}, the function $w_R(x,y)$ and integral $w_R \ast g$ are now scalar-valued, but the phenomenological motivation is similar to that of \eqref{convo1a}. Here, $d\in \mathbb{R}^{+}$ is the diffusion rate, and $\nu \in \mathbb{R}$ represents the advection coefficient. Equation \eqref{convo1b} can be interpreted as a model that describes the movement of a population influenced by the gradient of its non-local measure. Following the least energy principle, $w_R * g$ represents the energy density, and $\nabla (w_R \ast g)$ signifies the movement based on the energy gradient. When $\nu>0$, this reduces to the energy minimization problem \cite{carrillo2014}. 

Both models \eqref{convo1a} and \eqref{convo1b} have played a crucial role in the realm of ecological systems since the 1970s (see \cite{kawasaki1978, grunbaum1994, okubo2001}), and in the domain of cellular systems since the 1990s (see \cite{sekimura1999, gerisch2008}). These models, commonly referred to as aggregation equations, hold immense value due to their ability to capture self-organizing phenomena. By modeling the process of self-attraction among individuals, these equations enable dispersed groups to spontaneously organize into one or more aggregated groups.

It is important to highlight that these equations are not limited in their applications. They can also account for repulsive interactions, which promote enhanced dispersal. Furthermore, they can be employed to describe heterogeneous populations, where interactions between different populations vary. In such cases, these equations can be combined with more intricate models to capture the complexity of the system. As a result, there have been multi-species adaptations of both \eqref{convo1a} and \eqref{convo1b} to cater to these diverse scenarios.

The multi-species models are given as follows: Let $x\in\mathbb{T}^n$ denote the state variable, $t\in [0,\infty)$ denote the time variable, and $\mathbb{T}^{n}=[-L_1,L_1]\times [-L_2,L_2] \times \cdots \times [-L_n,L_n]$ be the $n$-torus with periodic boundaries. The $N$-species non-local population models we are interested in take the forms:
\begin{equation}\label{multispecies1}
    \partial_{t}u_{i}=d_{i}\Delta u_{i} - \sum\limits^{N}_{j=1}\mu_{ij}\nabla \cdot (u_i \mathbf{k}_{ij}\ast f_{ij}), \quad \mathbf{k}_{ij}\ast f_{ij}=\int_{\mathbb{T}^{n}} \mathbf{k}_{ij}(x,y) f_{ij}(\mathbf{u}(y,t))dy,
\end{equation}
\begin{equation}\label{multispecies2}
    \partial_{t}u_{i}=d_{i}\Delta u_{i}- \sum\limits^{N}_{j=1}\nu_{ij}\nabla \cdot [u_i \nabla (w_{ij}\ast g_{ij})], \quad w_{ij}\ast g_{ij}\int_{\mathbb{T}^{n}} w_{ij}(x,y) g_{ij}(\mathbf{u}(y,t))dy,
\end{equation}
where $i=1,\dots,N,$ $\mathbf{u}(x,t)= \big( u_1(x,t), u_2(x,t), \dots, u_N(x,t)\big),$ $u_i$ denotes the density distribution of $i$-th population. $d_i$ represents the diffusion coefficient and $\mu_{ij}, \nu_{ij}$ are the advection coefficients. 

The generalization of biological models to accomodate for multiple species offers a broader perspective and a more comprehensive understanding of ecological and cellular systems. By transforming \eqref{convo1a} and \eqref{convo1b} into \eqref{multispecies1} and \eqref{multispecies2}, researchers can explore the intricate dynamics and interactions between different populations. From a biological perspective, the advantage of such multi-species consideration is twofold. 
Firstly, it allows for the examination of inter-species relationships and their impact on the overall system behavior \cite{van2014}. This includes studying predator-prey interactions, competition for resources, symbiotic relationships, and other forms of ecological or cellular interactions. By considering multiple species, the models can capture the complexity and interdependence of these relationships, providing insights into the stability, coexistence, or potential disruptions within the system. 
Secondly, the generalization to multi-species enables the exploration of emergent properties and collective behaviors that arise from interactions between different species \cite{aschenbrenner2016, elias2012}. These emergent properties may include pattern formation, spatial organization, synchronization, or cooperative behaviors. By incorporating multiple species, the models can capture the synergistic effects and feedback mechanisms that give rise to these emergent phenomena. 

An interesting observation to note is that the diffusion coefficient $\mathbf{d}$ in equations \eqref{multispecies1} and \eqref{multispecies2} can represent two distinct scenarios: a single diffusion rate for all species ($\mathbf{d}=(d,\dots,d)$) or different diffusion rates assigned to each species ($\mathbf{d}=(d_1,\dots,d_N)$). Both of these settings have their own physical interpretations. For instance, consider several bird species in a forest with similar flying capabilities and habitats. They may have comparable diffusion rates, leading to similar movements, dispersal patterns, and coverage of distances and areas over time. On the other hand, a fast-swimming fish and a slow-moving crustacean living in the same ocean area have different diffusion rates, reflecting their distinct abilities to disperse and occupy different areas within their marine environment.

Therefore, the use of multi-species generalization in biological models enhances our understanding of the intricate dynamics and interactions within ecological and cellular systems. It allows for a more realistic representation of the natural world, providing valuable insights into the stability, coexistence, and emergent properties of diverse species within these systems. In this paper, we mainly consider the inverse problems of determining the diffusion coefficients $\mathbf{d}(x):= \{d_1(x),d_2(x),\dots, d_N(x)\}$ and integral interaction kernels $\mathbf{k}(x):=\{\mathbf{k}_{ij}(x)\}, \mathbf{w}(x):= \{w_{ij}(x)\}$, all depending only on $x$ and independent of $t$, in models based on \eqref{multispecies1} and \eqref{multispecies2}, which are of the following forms:
\begin{equation}\label{adoptedmodel1}
\begin{cases}
    \partial_{t}u_{i}=d_{i}\Delta u_{i} + \nabla \cdot \big( h(u_i) \sum\limits^{N}_{j=1}\mu_{ij}(\mathbf{k}_{ij}\ast u_j)\big),  &\  \text{in}\  Q,\\ 
    u_i(x,0)=f_i(x), &\  \text{in}\  \mathbb{T}^{n},\\
\end{cases}
\end{equation}
where $ \mathbf{k}_{ij}\ast u_j=\int_{\mathbb{T}^{n}} \mathbf{k}_{ij}(x-y) u_j(y)dy,$ and 
\begin{equation}\label{adoptedmodel2}
\begin{cases}
    \partial_{t} u_{i}=d_{i}\Delta u_{i} + \nabla \cdot \big( h(u_i) \sum\limits^{N}_{j=1} \nu_{ij}\nabla (w_{ij}\ast u_j) \big),   &\  \text{in}\  Q,\\ 
    u_i(x,0)=f_i(x), &\  \text{in}\  \mathbb{T}^{n},\\
\end{cases}
\end{equation}
where $ w_{ij}\ast u_j(x)=\int_{\mathbb{T}^{n}} w_{ij}(x-y) u_j(y)dy$.

The different species are represented by $i=1,2,\dots, N$, and the function $h(u_i)=u_i$ if $u_i \geq 0$ and $h(u_i)=0$ if $u_i <0$ is essential to ensure that the model is physically meaningful, as we will discuss in Section \ref{tech}. The interaction kernel $\mathbf{k}_{ij}$ describes the non-local sensing of species $j$ by species $i$ and is differentiable, and the interaction potential $w_{ij}$ describes the non-local sensing of species $j$ by species $i$ and is a twice-differentiable function such that $w_{ij} \geq 0$ is a non-increasing function of $\vert x \vert$ with $\nabla w_{ij}$ bounded. 

To recover the unknown coefficients, we assume that we can measure $\mathbf{u}(x,T)$ as well as $\mathbf{u}(x,t)$ in an accessible region $x\in\omega\Subset\mathbb{T}^{n}$ for all possible initial configurations $\mathbf{u}(x,0),$ that is, we have data from the maps: 
\begin{equation}\label{mmap1}
    \mathcal{M}^{+}_{\mathbf{d}}\big( \mathbf{u}(x,0) \big)=\mathbf{u}(x,t), \quad x\in\omega\Subset\mathbb{T}^{n}, t\in(0,T) 
\end{equation}
and
\begin{equation}\label{mmap2}
    \mathcal{M}^{+}_{\pmb{\mu},\mathbf{k}}\big( \mathbf{u}(x,0) \big)=\mathbf{u}(x,T),\quad \mathcal{M}^{+}_{\pmb{\nu},\mathbf{w}}\big( \mathbf{u}(x,0) \big)=\mathbf{u}(x,T),
\end{equation}
where the sign `$+$' signifies that the data are associated with the non-negative solutions of the models \eqref{adoptedmodel1} and \eqref{adoptedmodel2}.

Here are the questions we are interested in:

$\textbf{Case 1.}$  Let $\mathcal{M}^{+}_{\mathbf{d}^1}$ and $\mathcal{M}^{+}_{\mathbf{d}^2}$ be the measurement maps corresponding to the unknown coefficients $\mathbf{d}^1$, $\mathbf{d}^2$ in \eqref{adoptedmodel1} or \eqref{adoptedmodel2} respectively. If $\mathcal{M}^{+}_{\mathbf{d}^1}=\mathcal{M}^{+}_{\mathbf{d}^2}$ holds under appropriate assumptions, can we establish the unique identifiability conclusion $\mathbf{d}^1=\mathbf{d}^2$?

$\textbf{Case 2.}$ Let $\mathcal{M}^{+}_{\pmb{\mu}^1,\mathbf{k}^1}$ and $\mathcal{M}^{+}_{\pmb{\mu}^2,\mathbf{k}^2}$ be the measurement maps  corresponding to the unknown coefficients $(\pmb{\mu}^1,\mathbf{k}^1)$ and $(\pmb{\mu}^2,\mathbf{k}^2)$ in \eqref{adoptedmodel1} respectively. If $\mathcal{M}^{+}_{\pmb{\mu}^1,\mathbf{k}^1}=\mathcal{M}^{+}_{\pmb{\mu}^2,\mathbf{k}^2}$ holds under appropriate assumptions, can we establish the unique identifiability conclusion for the advection coefficients $\pmb{\mu}^1=\pmb{\mu}^2$? Can we give the same conclusion for the integral interaction kernels $\mathbf{k}^1$ and $\mathbf{k}^2$?

$\textbf{Case 3.}$ Let $\mathcal{M}^{+}_{\pmb{\nu}^1,\mathbf{w}^1}$ and $\mathcal{M}^{+}_{\pmb{\nu}^2,\mathbf{w}^2}$ be the measurement maps  corresponding to the unknown coefficients $(\pmb{\nu}^1,\mathbf{w}^1)$ and $(\pmb{\nu}^2,\mathbf{w}^2)$ in \eqref{adoptedmodel2} respectively. If $\mathcal{M}^{+}_{\pmb{\nu}^1,\mathbf{w}^1}=\mathcal{M}^{+}_{\pmb{\nu}^2,\mathbf{w}^2}$ holds under appropriate assumptions, can we establish the unique identifiability conclusion for the advection coefficients $\pmb{\nu}^1=\pmb{\nu}^2$? Can we establish the same conclusion for the integral interaction kernels $\mathbf{w}^1=\mathbf{w}^2$?

One will see later that the mission for recovering diffusion rate $\mathbf{d}$ are alike in both models, so we summarize them together in the first case. We also focus on the unique identifiability issues for each model respectively, specifically, whether we can recover the advection coefficient $\pmb{\mu}$ or the integral kernel $\mathbf{k}$ in \eqref{adoptedmodel1}, and whether we can recover the advection coefficient $\pmb{\nu}$ or the integral kernel $\mathbf{w}$ in \eqref{adoptedmodel2}. The main results will be given in detail in Section \ref{mainrs}.

\subsection{Technical developments}\label{tech}
Inverse problems related to parabolic and hyperbolic equations have garnered significant attention in recent years, as evidenced by notable works such as \cite{isakov1993, pilant1986, lin2022simu} in the context of parabolic-type equations, and \cite{lin2024, lassas2018, wang2019} in the realm of hyperbolic-type equations. Given the diverse range of physical systems that can be described by nonlinear PDEs, it is pertinent to consider the inverse problems associated with these applications. One prominent example of such applications is in mean field game theory \cite{klibanov2023, klibanov2023coefficient, liu2023ipmfg,liu2023simu,DL23mfg,ren2023unique}, which has a distinct physical context but involves coupled nonlinear PDEs that bear resemblance to the ones encountered in mathematical biology models. Unique identifiability results are established for either the running cost or the terminal cost based on knowledge of the total cost in \cite{liu2023ipmfg,DL23mfg}, assuming the Lagrangian represents the kinetic energy. Inverse boundary problems for the mean field game system are also considered in \cite{liu2024inverse,liu2023simu}.

Another field in which inverse problems find extensive application is population ecology, where numerous parabolic PDE systems are employed to explain various phenomena in nature. The first such work was \cite{liulo2024}, wherein the authors considered the non-negativity of solutions for the biological inverse problem, using a linearization method called high-order variation. Under the same restriction, \cite{dingliu2023} proposed and studied several inverse problems for identifying unknown coefficients in a class of biological models, by making use of averaged-out boundary data. 
Later, in \cite{lilo2023}, the authors investigated the inverse problem of determining the coefficients of interaction terms within Lotka-Volterra models, with support from boundary observations of non-negative solutions. The results extend to Holling-Tanner type models as well as models reflecting the hydra effect. Additionally, \cite{lilo2024simu} presents an attempt to address chemotaxis models within the context of inverse problems.

For these biological applications, the assurance of the non-negativity constraint is crucial, since species populations or cellular densities cannot be negative. And the high-order variation method, which is first introduced in \cite{liu2023simu, liulo2024}, has been widely used in recent works to guarantee that this constraint is fulfilled. The main idea of this method relies on a special setting for Taylor expansion,
\begin{equation}\label{characv}
\mathbf{u}(x;\varepsilon)=\sum\limits^{\infty}_{l=1}\varepsilon^{l}\mathbf{f}_{l} \  \  \text{on}  \,  \   \mathbb{T}^{n} \  \  \text{for} \,   \  \mathbf{f}_1 \geq 0,
\end{equation}
where $\varepsilon$ is a small positive variable, and the dominant term $\mathbf{f}_1$ determines and ensures the non-negativity of $\mathbf{u}(x; \varepsilon)$. No restrictions for the positivity of $\mathbf{f}_i, i \geq 2$ are required.
Similar to \cite{liulo2024}, we will carry out the linearization around $\mathbf{0}$ in the recovery of the advection coefficients and the integral interaction kernels. More details will be given in Section \ref{hovaria}.

Moreover, in order to make our model more versatile, we modify the model slightly in the form of \eqref{adoptedmodel1} and \eqref{adoptedmodel2}, while accounting for the non-negativity constraint. We have included a function $h(\cdot)$, which plays a role in avoiding the possibility of a negative solution for the models, in the form of 
\[    \partial_{t}u_{i}=d_{i}\Delta u_{i} + \nabla \cdot \big( h(u_i) \sum\limits^{N}_{j=1}\mu_{ij}(\mathbf{k}_{ij}\ast u_j)\big),  \quad
    \partial_{t} u_{i}=d_{i}\Delta u_{i} + \nabla \cdot \big( h(u_i) \sum\limits^{N}_{j=1} \nu_{ij}\nabla (w_{ij}\ast u_j) \big),   \  \text{in}\  Q.\] 
    When the solutions are non-negative, $h(u) =u$, and the models  \eqref{adoptedmodel1} and \eqref{adoptedmodel2} reduce to
\begin{equation}\label{origin1}
    \partial_{t}u_{i}=d_{i}\Delta u_{i} + \nabla \cdot \big( u_i \sum\limits^{N}_{j=1}\mu_{ij}(\mathbf{k}_{ij}\ast u_j)\big),
\end{equation}
\begin{equation}\label{origin2}
    \partial_{t} u_{i}=d_{i}\Delta u_{i} + \nabla \cdot \big( u_i \sum\limits^{N}_{j=1} \nu_{ij}\nabla (w_{ij}\ast u_j) \big), 
\end{equation}
which are of similar forms as \eqref{multispecies1} and \eqref{multispecies2}. When $u<0$, $h(u) = 0$ and the main parts of the models vanish, reducing to the simple classic heat equation
\[\partial_{t} u_{i}=d_{i}\Delta u_{i}.\]

For these biological models, we are mainly concerned with the issue of unique identifiability of the associated inverse problems, as discussed in the previous subsection. In terms of unique identifiability, recovering the diffusion coefficient $\mathbf{d}$ poses different challenges compared to recovering the advection coefficients and integral interaction kernels. In wave equations, the unique identifiability of the coefficient preceding the Laplace operator, which represents the wave speed in a medium, has been extensively studied. In these case, the Laplace transform is commonly applied to the entire equation to solve the problem \cite{liu2015determining}. This transform results in a Helmholtz equation, which has well-known solutions allowing for the unique determination and recovery of the wave speed. However, this direct approach does not work for the equations \eqref{adoptedmodel1} and \eqref{adoptedmodel2}. Furthermore, the standard method for elliptic operators using the Liouville transform is also not applicable to parabolic operators due to the absence of easy cancellation of factors after the transformation, particularly involving the time derivative term.

 In this regard, the method proposed in \cite{goncharsky2024} offers us a new idea, where the nonlinear coefficient inverse problem is found to be equivalent to solving a linear Fredholm integral equation of the first kind. Building upon this idea, in this work, we develop a groundbreaking technique, which we call the \emph{transformative asymptotic method}, to address this problem.  This method offers a fresh and innovative perspective on solving the problem and is capable of recovering scenarios with uniform diffusion rates for all species ($\mathbf{d}=(d,\dots,d)$) or varying diffusion rates for each species ($\mathbf{d}=(d_1,\dots,d_N)$). It is important to note, however, that this method is specifically applicable only in two-dimensional spaces. Nevertheless, our technique can be easily extended to encompass a broader range of operators, including operators in divergence form, through a straightforward expansion of the operator. 

Simultaneously, it is important to acknowledge that our work aims to recover functions $\mathbf{k}$ and $\mathbf{w}$ that are involved in complex convolution and difference operators. Furthermore, the main part of the models is also of second-order. These unique forms introduce disturbances and set our models apart from those discussed in \cite{lilo2023, lilo2024simu, liu2023simu}. However, by leveraging the properties of convolution in our proof, we are able to effectively address this challenge.

\sloppy Another point worth mentioning is that our work is the first research addressing the unique identifiability issues for inverse problems of mathematical biology models based on multi-population systems, for any finite number of populations. This aligns more closely with the physical background of complex ecosystems which are composed of numerous species. For instance, ecosystems such as rainforests, deserts, and snow-capped mountains composed of multiple populations intricately interacting together, and it is impossible and physically meaningless to isolate and study individual organisms or pairs of organisms within these environments. 

In summary, we outline the major contributions of this work:
\begin{enumerate}[(1)]
    \item Firstly, we considered the inverse problem for a complex high order multi-population aggregation system, which has numerous applications in biological ecosystems. Furthermore, the model involves a function $h(u)$ which ensures non-negativity of the solutions.
    \item Secondly, we derived the unique identifiability result for the diffusion coefficients in the model, using a completely novel method known as the transformative asymptotic technique. To the best of our knowledge, this is the first work that attempts to recover the diffusion coefficient in parabolic-type equations with periodic boundary conditions.
    \item Thirdly, we obtained uniqueness results for the advection coefficients and interaction terms, which are involved in the higher order part of the operator. We made use of the high-order variation method, which has only been introduced recently, to ensure the non-negativity of the solutions, thereby maintaining that our results are physically realistic.
\end{enumerate}

The remaining sections of the paper are structured as follows. Section \ref{prelims} provides an introduction to the notation used, presents preliminary results related to the forward problems, and states the main results of the paper. In Section \ref{hovaria}, we offer a concise overview of the high-order variation method. The recovery of the diffusion coefficient $\mathbf{d}$ is addressed in Section \ref{Proof_d}, while the recovery for the advection coefficients and integral interaction kernels are presented in Section \ref{Proof_muk} and Section \ref{Proof_nuw}, respectively for the models \eqref{adoptedmodel1} and \eqref{adoptedmodel2}, as well as their respective applications.

{\centering \section{PRELIMINARIES AND STATEMENT OF MAIN RESULTS}  \label{prelims} }
\subsection{Notations and basic settings}
We consider the multi-population mathematical biology models on the $n$-dimensional torus $\mathbb{T}^n,$ whereon any function defined should be $(1,1,\dots,1)$-periodic with respect to $x.$ To be specific, for each $x_i\,  (i=1,2,\dots,n),$ the function would be 1-periodic.

\sloppy Let $\mathbb{N}$ be the set of all non-negative integers. For a $n$-dimensional multi-index $\mathbf{\alpha}=(\alpha_1,\dots,\alpha_{n})\in \mathbb{N}^{n}$ and $x=(x_1,\dots,x_n)\in \mathbb{R}^n,$ we have the following notations:
\[D^{\mathbf{\alpha}}=\partial^{\alpha_1}_{x_1}\partial^{\alpha_2}_{x_2}\cdots \partial^{\alpha_n}_{x_n}, \  \mathbf{\alpha}!= \alpha_1 ! \alpha_2 !\cdots \alpha_n !,\  \vert \mathbf{\alpha} \vert=\sum\limits^{n}_{i=1}\alpha_{i}.\]

Then for $k\in\mathbb{N}$ and $\alpha \in (0,1),$ we define the H\"older space $C^{k+\alpha}(\mathbb{T}^n)$ as the subspace of $C^{k}(\mathbb{T}^n),$ in which the function has $k$ derivatives which are $\alpha$-H\"older continuous for all $\vert \mathbf{\alpha} \vert \leq k.$ The norm is defined as
\begin{equation}\label{define_norm1}
    \Vert u \Vert_{C^{k+\alpha}(\mathbb{T}^n)} := \sum\limits_{\vert \mathbf{\alpha} \vert \leq k} \Vert D^{\mathbf{\alpha}} u \Vert_{L^{\infty}(\mathbb{T}^n)}+\sum\limits_{\vert \mathbf{\alpha} \vert=k} \mathop{\text{sup}}\limits_{x \neq y}\frac{\vert D^{\mathbf{\alpha}} u(x)-D^{\mathbf{\alpha}} u(y)\vert}{\vert x-y \vert ^{\alpha}}.
\end{equation}

\sloppy For functions depending on both the time and space variables,  we define its space as $C^{k+\alpha, \frac{k+\alpha}{2}}(\mathbb{T}^n).$ Denote $Q:=\mathbb{T}^n \times [0,T],$ we say that a function $u \in  C^{k+\alpha, \frac{k+\alpha}{2}}(Q)$ if $D^{\mathbf{\alpha}} \partial^{j}_{t} u$ exists and is $d$-H\"older continuous with exponent $\alpha$ in the space variable and exponent $\frac{k+\alpha}{2}$ in the time variable for all $\mathbf{\alpha} \in \mathbb{N}^n,$ $j \in \mathbb{N},$ with $\vert \mathbf{\alpha} \vert + 2j \leq k.$ The norm of $C^{k+\alpha, \frac{k+\alpha}{2}}(Q)$ is defined as
\begin{equation}\label{define_norm2}
    \Vert u \Vert_{C^{k+\alpha, \frac{k+\alpha}{2}}(Q)} := \sum\limits_{\vert \mathbf{\alpha}+2j \vert \leq k} \Vert D^{\mathbf{\alpha}} \partial^{j}_{t} u \Vert_{L^{\infty}(Q)}+\sum\limits_{\vert \mathbf{\alpha} \vert+2j=k} \mathop{\text{sup}}\limits_{(x,t) \neq (y,s)}\frac{\vert u(x,t)- u(y,s)\vert}{\vert x-y \vert ^{\alpha}+\vert t-s \vert^{\frac{\alpha}{2}}}.
\end{equation}

Since the functions we will be treating are vector-valued in multi-population systems, the H\"older norm of $\mathbf{u}=(u_1,u_2,\dots,u_N)$ is defined as 
\begin{equation}\label{define_norm3}
     \begin{aligned}
         \Vert \mathbf{u} \Vert_{C^{k+\alpha}(\mathbb{T}^n)} & := \sum\limits^{N}_{i=1} \Vert u_i \Vert_{C^{k+\alpha}(\mathbb{T}^n)},\\
    \Vert \mathbf{u} \Vert_{C^{k+\alpha, \frac{k+\alpha}{2}}(Q)} & := \sum\limits^{N}_{i=1} \Vert u_i \Vert_{C^{k+\alpha, \frac{k+\alpha}{2}}(Q)}.
     \end{aligned} 
\end{equation}

\subsection{Well-posedness of the forward problem}\label{Adc}
This subsection aims to discuss the well-posedness of multiple initial-boundary value problems for semi-linear parabolic equations. First, we present a well-posed situation for \eqref{adoptedmodel1}.
\begin{lem}[Theorem 5.1 of \cite{painter2023}]\label{wellposed_eco1}
Consider the system
\begin{equation}\label{in_lem1}
    \begin{cases}
        \partial_t \mathbf{u}= \mathbf{d}\Delta \mathbf{u} - \nabla \cdot \left(\pmb{\mu} \mathbf{u} \left(\frac{x}{|x|} \mathbf{K}(x)*\mathbf{u}(x,t)\right)\right),  &\  \text{in}\  Q,\\ 
        \mathbf{u}(x,0)=\mathbf{f}(x), &\  \text{in}\  \mathbb{T}^{n},\\
    \end{cases}
\end{equation}
where $\mathbf{d}=d I_N$ for constant $d$ and $N\times N$-identity matrix $I_N$. Assume:\par
$(\mathbf{L1})$ for $p \geq 1,$ let $f_i(x) \in X_{p}:= C^0(\mathbb{T}^n)\cap L^{\infty} (\mathbb{T}^n)\cap L^{p} (\mathbb{T}^n)$ be non-negative for $\mathbf{f}=(f_1,\dots,f_N)$;\par
$(\mathbf{L2})$ $\mathbf{K} \in L^1(\mathbb{T}^n),$ $\mathbf{K} \geq 0$.\\
Then there exists a unique, classical, global solution $\mathbf{u}=(u_1,\dots,u_N)$ such that, for all $i=1,\dots,N$,
\[u_i \in C^0([0,\infty);X_p) \cap C^{2,1}(B_R(0) \times (0,\infty)),\]
for a small ball $B_R(0)$ about the origin.
\end{lem}

We observe that the assumptions and model stated here are slightly different from those given in \cite{painter2023}, where the authors consider a semilinear function $g(\mathbf{u})$ in the convolution product, which is approximately linear up to a certain bounded density and becomes zero beyond which. This condition finds its roots in the biological context, where it signifies the spatial limitations or saturation of receptors. It aligns with our natural expectations, as physical or biological constraints would naturally give rise to a bound. However, in this work, we only consider the solutions $\mathbf{u}$ near 0, as seen from the method of higher order variation. Therefore, it suffices for us to consider a specific case of the result of \cite{painter2023}, taking $\mathbf{u}$ directly in \eqref{in_lem1} instead of using $g(\mathbf{u})$.

Practically, this model is usually used to describe the cell–cell adhesion. The conditions outlined in Lemma \ref{wellposed_eco1} capture the phenomenon wherein cells within a certain proximity exhibit strong adhesive forces, while cells beyond this range experience minimal adhesion. This model is particularly valuable for simulating scenarios where cell adhesion is confined to specific physical boundaries or contact regions.

Similarly, there is a well-posedness result for the case of \eqref{adoptedmodel2}.
 \begin{lem}[Corollary 5.1 of \cite{painter2023}]\label{wellposed_eco2}
     Consider the model 
     \begin{equation}\label{in_lem2}
     \begin{cases}
         \partial_t \mathbf{u} =\mathbf{d} \Delta \mathbf{u} + \nabla \cdot (\mathbf{u} \nabla \pmb{\nu} (\mathbf{W}\ast \mathbf{u})), &\  \text{in}\  Q,\\ 
        \mathbf{u}(x,0)=\mathbf{f}(x), &\  \text{in}\  \mathbb{T}^{n},\\
     \end{cases}
     \end{equation}
     where $\pmb{\nu} >0$, $\mathbf{d}$ satisfies the same conditions as Lemma \ref{wellposed_eco1}, and the potential $W(\vert r \vert)$ is a function of the distance of the interaction $\vert r\vert =\vert y-x \vert.$ 
     
     Assume for all $i=1,\dots,N$:\par
     $(\mathbf{I1})$ for $p\geq 1,$ let $f_i(x) \in X_p:=C^0(\mathbb{T}^{n})\cap L^{\infty}(\mathbb{T}^{n})\cap L^p(\mathbb{T}^n) \cap C^{2+\alpha}(\mathbb{T}^{n})$ be non-negative;\par
     $(\mathbf{I2})$ $W_i(\vert r\vert)\in L^{\infty}$ and $W_i(\vert r \vert)$ has compact support inside a ball $B_R(0);$ \par
    $(\mathbf{I3})$ There exists a scalar function $w_i(|r|)$ such that $\nabla W_i(|r|)=\frac{r}{|r|}w_i(|r|)$ where $w_i(|r|) \in L^1(\mathbb{T}^{n})$ and $w_i(|r|) \geq 0.$\\
    Then  there exists a unique, classical, global solution $\mathbf{u}=(u_1,\dots,u_N)$ such that, for each $u_i$,
      \[u_i \in C^0([0,\infty);X_p)\cap C^{2,1}(B_r(0)\times (0,\infty)).\]
 \end{lem}

It is apparent that Lemma \ref{wellposed_eco2} can be easily derived from Lemma \ref{wellposed_eco1} by substituting $\frac{r}{|r|}\mathbf{K}(r)$ with $\nabla \mathbf{W}$. This substitution leads to the emergence of condition $(\mathbf{I3})$, which is notably more restrictive. This condition mandates that the drift is consistently directed towards the origin, with the origin representing the position of the probing individual. Consequently, the forces within a group are consistently attractive, giving rise to the observable phenomena exhibited by social animals. For example, hyenas within a clan tend to exhibit cohesive behavior, with individuals often staying close together and coordinating their actions. When hunting, hyenas work together in a coordinated manner to pursue and capture prey by communicating with vocalizations and visual signals. Under this arrangement, they enhance their hunting success, protect their territory, and strengthen their overall social structure. 

\begin{rmk}\label{rmk_wellposed}
It is challenging to arrive at a general conclusion regarding the well-posedness of the models \eqref{adoptedmodel1} and \eqref{adoptedmodel2}, and it is apparent that the aforementioned two lemmas represent two specific instances for them respectively. However, we are aware that solutions must exist for \eqref{adoptedmodel1} and \eqref{adoptedmodel2} based on their physical significance. In subsection \ref{mainrs}, we will make direct assumptions for these models and present our main theorems. Lemma \ref{wellposed_eco1} and Lemma \ref{wellposed_eco2} are treated as distinct applications of the unique identifiability problem, and their proofs are included at the end of the respective proofs for the two main models.
\end{rmk}

\subsection{Statement of main results} \label{mainrs}
Having discussed the associated forward problems, in this section, we present our main results for the inverse problems, which demonstrates that under generic conditions, we can uniquely recover the diffusion coefficients $\mathbf{d}$ from the measurement map $\mathcal{M}^{+}_{\mathbf{d}}$ for systems \eqref{adoptedmodel1} and \eqref{adoptedmodel2}. Additionally, for the system \eqref{adoptedmodel1}, we can recover the advection coefficient $\pmb{\mu}$ and the normalization constant of the integral kernel $\mathbf{k}$ from the measurement map $\mathcal{M}^{+}_{\pmb{\mu},\mathbf{k}}$. On the other hand, for the system \eqref{adoptedmodel2}, we are able to obtain a stronger result and establish the unique identifiability of the advection coefficient $\pmb{\nu}$ and the integral kernel $\mathbf{w}$ from the measurement map $\mathcal{M}^{+}_{\pmb{\nu},\mathbf{w}}$. Finally, we apply these results to verify the uniqueness of systems \eqref{in_lem1} and \eqref{in_lem2} at the end of this subsection.

We start by restating the  nonlinear system of equations of \eqref{adoptedmodel1} and \eqref{adoptedmodel2} for $l=1,2$ as:
\begin{equation}\label{sec2_3_ad1}
\begin{cases}
    \partial_{t}u^{l}_{i}=d^{l}_{i}\Delta u^{l}_{i} + \nabla \cdot \big( h(u^{l}_i) \sum\limits^{N}_{j=1}\mu^{l}_{ij}(\mathbf{k}^{l}_{ij}\ast u^{l}_j)\big),  &\  \text{in}\  Q,\\ 
    u^{l}_i(x,0)=f^{l}_i(x), &\  \text{in}\  \mathbb{T}^{n},
\end{cases}
\end{equation}
and

\begin{equation}\label{sec2_3_ad2}
\begin{cases}
    \partial_{t} u^{l}_{i}=d^{l}_{i}\Delta u^{l}_{i} + \nabla \cdot \big( h(u^{l}_i) \sum\limits^{N}_{j=1} \nu^{l}_{ij}\nabla (w^{l}_{ij}\ast u^{l}_j) \big),   &\  \text{in}\  Q,\\ 
    u^{l}_i(x,0)=f^{l}_i(x), &\  \text{in}\  \mathbb{T}^{n}.
\end{cases}
\end{equation}

First, we present the result for the diffusion coefficient $\mathbf{d}$.
\begin{thm}\label{mainthm_DR}
     Suppose $n=2$. Assume that the system \eqref{sec2_3_ad1} has a solution $u ^l_i\in C^{1+\frac{\alpha}{2},2+\alpha}(Q)$ for $i=1,\dots,N$ for $l=1,2$ respectively. Let $\mathcal{M}^{+}_{\mathbf{d}^{l}}$ be the associated measurement map for $l = 1, 2$ as defined in \eqref{mmap1}. If for any $\mathbf{f}^l \in C^{2+\alpha}(\mathbb{T}^2),$ one has  
$\mathcal{M}^{+}_{\mathbf{d}^1}(\mathbf{f}^1)=\mathcal{M}^{+}_{\mathbf{d}^2}(\mathbf{f}^2)$,
then it holds that $\mathbf{d}^1(x)=\mathbf{d}^2(x)$ for $x\in\omega\Subset\mathbb{T}^2$.

    Similarly, if the system \eqref{sec2_3_ad2} has a solution $u^l_i \in C^{1+\frac{\alpha}{2},2+\alpha}(Q)$ and $\mathcal{M}^{+}_{\mathbf{d}^{l}}$ is the associated measurement map for $l = 1, 2$, and for any $\mathbf{f}^l \in C^{2+\alpha}(\mathbb{T}^2),$ one has  
$\mathcal{M}^{+}_{\mathbf{d}^1}(\mathbf{f}^1)=\mathcal{M}^{+}_{\mathbf{d}^2}(\mathbf{f}^2)$,
then it holds that $\mathbf{d}^1(x)=\mathbf{d}^2(x)$ for $x\in\omega\Subset\mathbb{T}^2$.
\end{thm}

A natural corollary for this result is the following:
\begin{cor}\label{cor_DR}
     Suppose $n=2$. Assume that the system \eqref{sec2_3_ad1} has a solution $u^l_i \in C^{1+\frac{\alpha}{2},2+\alpha}(Q)$ for $i=1,\dots,N$ for $l=1,2$ respectively. Let $\mathcal{M}^{+}_{\mathbf{d}^{l}}$ be the associated measurement map for $l = 1, 2$ as defined in \eqref{mmap1}. If for any $\mathbf{f}^l \in C^{2+\alpha}(\mathbb{T}^2)$ and $\mathbf{d}^l$ analytic in $x$, one has  
$\mathcal{M}^{+}_{\mathbf{d}^1}(\mathbf{f}^1)=\mathcal{M}^{+}_{\mathbf{d}^2}(\mathbf{f}^2)$,
then it holds that $\mathbf{d}^1(x)=\mathbf{d}^2(x)$ for $x\in\mathbb{T}^2$.

    Similarly, if the system \eqref{sec2_3_ad2} has a solution $u^l \in C^{1+\frac{\alpha}{2},2+\alpha}(Q)$ and $\mathcal{M}^{+}_{\mathbf{d}^{l}}$ is the associated measurement map for $l = 1, 2$, and for any $\mathbf{f}^l \in C^{2+\alpha}(\mathbb{T}^2)$ and $\mathbf{d}^l$ analytic in $x$, one has  
$\mathcal{M}^{+}_{\mathbf{d}^1}(\mathbf{f}^1)=\mathcal{M}^{+}_{\mathbf{d}^2}(\mathbf{f}^2)$,
then it holds that $\mathbf{d}^1(x)=\mathbf{d}^2(x)$ for $x\in\mathbb{T}^2$.
\end{cor}

Next, for any $n\in\mathbb{N}$, we give the unique identifiability conclusion for the advection coefficient $\pmb{\mu}$ and the normalization constant $\mathbf{k}$:
\begin{thm}\label{mainthm_Inte1}
    Suppose that the system \eqref{sec2_3_ad1} has a solution $u^l_i \in C^{1+\frac{\alpha}{2},2+\alpha}(Q)$ for $i=1,\dots,N$. Let $\mathcal{M}^{+}_{\pmb{\mu}^{l},\mathbf{k}^{l}}$ be the measurement map associated for $l = 1, 2$ as defined in \eqref{mmap2}. 
    
    For any $\mathbf{f}^l \in C^{2+\alpha}(\mathbb{T}^{n})$, if $\mathcal{M}^{+}_{\pmb{\mu}^1,\mathbf{k}}(\mathbf{f}^1)=\mathcal{M}^{+}_{\pmb{\mu}^2, \mathbf{k}}(\mathbf{f}^2)$, then $\pmb{\mu}^1=\pmb{\mu}^2$, assuming that the normalization constant of $\mathbf{k}$ is known and non-zero and $\mathbf{d}$ is independent of $x$.  
    
    Conversely, suppose $\mathbf{d}$ is independent of $x$. If we have knowledge of the non-zero $\pmb{\mu}$, then for any $\mathbf{f}^l \in C^{2+\alpha}(\mathbb{T}^{n})$,  $\mathcal{M}^{+}_{\pmb{\mu}^1,\mathbf{k}^1}(\mathbf{f}^1)=\mathcal{M}^{+}_{\pmb{\mu}^2,\mathbf{k}^2}(\mathbf{f}^2)$ implies that the normalization constants of $\mathbf{k}^1$ and $\mathbf{k}^2$ are the same, i.e., 
    \[ \sum\limits_{j=1}^n\int_{\mathbb{T}^n}\frac{d}{dx_j}\mathbf{k}^{1}_{pq,j}(x-y)dy=\sum\limits_{j=1}^n\int_{\mathbb{T}^n}\frac{d}{dx_j}\mathbf{k}^{2}_{pq,j}(x-y)dy,\quad \forall p,q=1,\dots,N,\text{ in }Q.\]
\end{thm}

Similarly, we have the unique identifiability conclusion for the advection coefficient $\pmb{\nu}$ and the integral kernel $\mathbf{w}$:
\begin{thm}\label{mainthm_Inte2}
    Suppose that the system \eqref{sec2_3_ad2} has a solution $u^l_i \in C^{1+\frac{\alpha}{2},2+\alpha}(Q)$ for $i=1,\dots,N$. Let $\mathcal{M}^{+}_{\pmb{\nu}^{l},\mathbf{w}^{l}}$ be the measurement map associated for $l = 1, 2$ as defined in \eqref{mmap2}. 
    
    Then, for any $\mathbf{f}^l \in C^{2+\alpha}(\mathbb{T}^{n}),$ $\mathcal{M}^{+}_{\pmb{\nu}^1,\mathbf{w}^1}(\mathbf{f}^1)=\mathcal{M}^{+}_{\pmb{\nu}^2 \mathbf{w}^2}(\mathbf{f}^2)$ implies that $\pmb{\nu}^1=\pmb{\nu}^2$, assuming that $\mathbf{w}$ is known and non-trivial. 
    
    Conversely, if we have knowledge of $\pmb{\nu}$, then for any $\mathbf{f}^l \in C^{2+\alpha}(\mathbb{T}^{n}),$ $\mathcal{M}^{+}_{\pmb{\nu}^1,\mathbf{w}^1}(\mathbf{f}^1)=\mathcal{M}^{+}_{\pmb{\nu}^2,\mathbf{w}^2}(\mathbf{f}^2)$ implies that $\mathbf{w}^1=\mathbf{w}^2$ in $Q$.
\end{thm}

\begin{rmk}
    It is important to note that symmetry is not required for $\pmb{\mu}, \pmb{\nu}, \mathbf{k}$ or $\mathbf{w}$. Also, the recovery of $\pmb{\mu}$ and $\mathbf{k}$, and $\pmb{\nu}$ and $\mathbf{w}$ are not simultaneous.
\end{rmk}

Moreover, in the specific cases outlined in Lemma \ref{wellposed_eco1} and Lemma \ref{wellposed_eco2} where well-posedness holds, the above theorems hold, and we have the following results:
\begin{cor} \label{appthm_case1}
        Consider the system
\begin{equation}\label{incon_lem1}
    \begin{cases}
        \partial_t \mathbf{u}^l= \mathbf{d}^l\Delta \mathbf{u}^l -\nabla \cdot \left( \pmb{\mu}^l  \mathbf{u}^l \left(\frac{x}{|x|} \mathbf{K}^l(x)*\mathbf{u}^l(x,t)\right) \right),  &\  \text{in}\  Q,\\ 
        \mathbf{u}^l(x,0)=\mathbf{f}^l(x), &\  \text{in}\  \mathbb{T}^{n},\\
    \end{cases}
\end{equation}
for $l=1,2$ which satisfy the conditions $(\mathbf{L1, L2}).$ 

Let $\mathcal{M}^{+}_{\mathbf{d}^l}$ and $\mathcal{M}^{+}_{\pmb{\mu}^l,\mathbf{K}^l}$ be the measurement maps associated to \eqref{incon_lem1} for $l=1,2$ respectively. The unique identifiability conclusions are as follows:\\
$\mathbf{(I)}$ If 
$\mathcal{M}^{+}_{\mathbf{d}^1}(\mathbf{f}^1)=\mathcal{M}^{+}_{\mathbf{d}^2}(\mathbf{f}^2)$, then $\mathbf{d}^1=\mathbf{d}^2$.\\
$\mathbf{(II)}$ Given that the normalization constant for $\mathbf{K}$ is known and non-zero, if 
$\mathcal{M}^{+}_{\pmb{\mu}^1,\mathbf{K}^1}(\mathbf{f}^1)=\mathcal{M}^{+}_{\pmb{\mu}^2,\mathbf{K}^2}(\mathbf{f}^2)$,
then it holds that $\pmb{\mu}^1=\pmb{\mu}^2$.\\
$\mathbf{(III)}$ Given an advection coefficient $\pmb{\mu}$, if 
$\mathcal{M}^{+}_{\pmb{\mu}^1,\mathbf{K}^1}(\mathbf{f}^1)=\mathcal{M}^{+}_{\pmb{\mu}^2,\mathbf{K}^2}(\mathbf{f}^2)$,
then it holds that the normalization constants for $\mathbf{K}^1$ and $\mathbf{K}^2$ are identical.
\end{cor}

\begin{cor}\label{appthm_case2}
        Suppose $\mathbf{W}^{l}$ $(l=1,2)$ satisfy $\mathbf{(I_2,I_3)}.$ Consider the systems
        \begin{equation}\label{appthm_eq}
     \begin{cases}
         \partial_t \mathbf{u}^l =\mathbf{d}^l \Delta \mathbf{u}^l + \nabla \cdot (\mathbf{u}^l \nabla \pmb{\nu}^l (\mathbf{W}^l \ast \mathbf{u}^l)), &\  \text{in}\  Q,\\ 
        \mathbf{u}^l(x,0)=\mathbf{f}^l(x), &\  \text{in}\  \mathbb{T}^{n},\\
     \end{cases}
     \end{equation}
    where $\mathbf{f}^l$ satisfies $\mathbf{(I_1)}$. 
     Let $\mathcal{M}^{+}_{\mathbf{d}^l}$ and $\mathcal{M}^{+}_{\pmb{\nu}^l,\mathbf{W}^l}$ be the measurement maps associated to \eqref{appthm_eq} for $l=1,2$ respectively. The unique identifiability conclusions are as follows:\\
$\mathbf{(I)}$ If 
$\mathcal{M}^{+}_{\mathbf{d}^1}(\mathbf{f}^1)=\mathcal{M}^{+}_{\mathbf{d}^2}(\mathbf{f}^2)$, then it follows that $\mathbf{d}^1=\mathbf{d}^2$.\\
$\mathbf{(II)}$ Given the knowledge of the non-trivial function $\mathbf{W}$, if 
$\mathcal{M}^{+}_{\pmb{\nu}^1,\mathbf{W}^1}(\mathbf{f}^1)=\mathcal{M}^{+}_{\pmb{\nu}^2,\mathbf{W}^2}(\mathbf{f}^2)$,
then it holds that $\pmb{\nu}^1=\pmb{\nu}^2$.\\
$\mathbf{(III)}$ Given a non-zero advection coefficient $\pmb{\nu}$, if 
$\mathcal{M}^{+}_{\pmb{\nu}^1,\mathbf{W}^1}(\mathbf{f}^1)=\mathcal{M}^{+}_{\pmb{\nu}^2,\mathbf{W}^2}(\mathbf{f}^2)$,
then it holds that $\mathbf{W}^1=\mathbf{W}^2$.
\end{cor}

{\centering \section{HIGH-ORDER VARIATION METHOD}  \label{hovaria} }
In this section, we discuss the high-variation method to address the positivity constraint on the biology models. As we will see later from the proofs of the theorems, we only need to execute the linearization process to the second-order in order to uniquely recover the integral kernels $\mathbf{k}$ and $\mathbf{w}$. 
Furthermore, the high-order variation method is a relatively novel approach aimed at guaranteeing non-negativity in physical models and simplifying nonlinear models into linear forms. This work represents the first attempt in employing the multi-variation technique for interacting population systems, to our best knowledge. An intriguing aspect is that the incorporation of the function $h(\mathbf{u})$ provides an alternative means of preserving the non-negativity of the solutions, while also allowing for the direct application of the multi-variation technique in the models.

Consider the systems \eqref{adoptedmodel1} and \eqref{adoptedmodel2}. For any known solution $\mathbf{u}_0$, consider a small-enough positive constant $\varepsilon$. Then, we can expand the function $\mathbf{f}(x,t;\varepsilon)$ as:
\begin{equation}\label{highvarygeneral}
    \mathbf{f}(x;\varepsilon)=\mathbf{u}_{0}+\varepsilon \mathbf{f}_{1}(x)+\frac{1}{2}\varepsilon^2 \mathbf{f}_{2}(x)+ \tilde{\mathbf{f}}(x;\varepsilon),
\end{equation}
where $\mathbf{f}_{1}, \mathbf{f}_{2} \in [C^{2+\alpha}(\mathbb{T}^n)]^N$, and $\tilde{\mathbf{f}}(x;\epsilon)$ satisfies 
\[\frac{1}{|\varepsilon|^3}\norm{\tilde{\mathbf{f}}(x;\epsilon)}_{[C^{2+\alpha}(\mathbb{T}^n)]^N}=\frac{1}{|\varepsilon|^3}\norm{\mathbf{f}(x;\varepsilon)-\mathbf{u}_{0}-\varepsilon \mathbf{f}_{1}(x)-\frac{1}{2}\varepsilon^2 \mathbf{f}_{2}(x)}_{[C^{2+\alpha}(\mathbb{T}^n)]^N}\to0
\]
uniformly in $\varepsilon$. 
When $\mathbf{u}_{0}=\mathbf{0},$ we ask $\mathbf{f}_{1} \geq 0,$ thus $\mathbf{f}$ maintains non-negativity as $\varepsilon \to 0.$ 

Now, we define the first-order variation form for each $i=1,\dots,N$. Let $S$ be the solution operator of \eqref{adoptedmodel1} or \eqref{adoptedmodel2} with respect to $\mathbf{f}$. We assume that there exists a bounded linear operator $A$ from $C^{2+\alpha}(\mathbb{T}^n)$ to $[C^{1+\frac{\alpha}{2},2+\alpha}(Q)]^N$ such that
\begin{equation}\label{highvaryIorder}
	\lim\limits_{\norm{\mathbf{f}}_{C^{2+\alpha}(\mathbb{T}^n)}\to0}\frac{\|S(\mathbf{f})-S(\mathbf{u}_0)- A(\mathbf{f})\|_{[C^{1+\frac{\alpha}{2},2+\alpha}(Q)]^N}}{\norm{\mathbf{f}}_{C^{2+\alpha}(\Omega)}}=0.
\end{equation}  
Then for fixed $\mathbf{f}_1$, $A(\mathbf{f})|_{\varepsilon=0}$ is the solution map for the first-order variation system:
\begin{equation}\label{highvaryIsys}
  \begin{cases}
    \partial_{t}u^{(I)}_{i}-d_{i}\Delta u^{(I)}_{i}=0, & \quad \text{in}\  Q,\\ 
    u^{(I)}_{i}(x,0)= f_{i,1}(x), & \quad  \text{in} \ \mathbb{T}^{n}, 
  \end{cases}
\end{equation}
where $i=1,2,\dots,N$, for both \eqref{adoptedmodel1} and \eqref{adoptedmodel2}. Here, we define 
\begin{equation}
\mathbf{u}^{(I)}(x,t) = (u^{(I)}_1(x,t),\dots,u^{(I)}_N(x,t)):=A(\mathbf{f})|_{\varepsilon=0}.
\end{equation}
For notational simplicity, we write 
\begin{equation}
    u_i^{(I)}(x,t):=\partial_{\varepsilon}u_i(x,t;\varepsilon) \vert_{\varepsilon=0}.
\end{equation}
These first order linearization equations allow us to recover the unknown diffusion coefficients $\mathbf{d}$, by examining the corresponding systems for each species independently.

Next, we define the second-order variation for $\mathbf{u}=(u_1,\dots,u_N)$ as:
\begin{equation}\label{highvaryIIorder}
    u_i^{(II)}:=\partial^{2}_{\varepsilon}u_i \vert_{\varepsilon=0}\quad \text{ for }i=1,\dots,N.
\end{equation}

And the second-order variation system of \eqref{adoptedmodel1} is:
\begin{equation}\label{highvaryII1}
\begin{cases}
    \partial_{t} u^{(II)}_{i}= d_{i}\Delta u^{(II)}_{i}-\nabla \cdot \left[u^{(I)}_{i} \sum\limits^{N}_{j=1}\mu_{ij} (\mathbf{k}_{ij} \ast u^{(I)}_{j}) \right], & \quad \text{in}\  Q,\\ 
    u^{(II)}_{i}(x,0)= 2f_{i,2}(x), & \quad  \text{in} \ \mathbb{T}^{n}, 
  \end{cases}
\end{equation}
for $i=1,2,\dots,N.$ 

Meanwhile, we have the second-order variation system of \eqref{adoptedmodel2} as:
\begin{equation}\label{highvaryII2}
\begin{cases}
    \partial_{t} u^{(II)}_{i}=d_{i}\Delta u^{(II)}_{i}-\nabla \cdot \left[u^{(I)}_{i} \sum\limits^{N}_{j=1}\nu_{ij} \nabla (w_{ij} \ast u^{(I)}_{j})\right], & \quad \text{in}\  Q,\\ 
    u^{(II)}_{i}(x,0)= 2f_{i,2}(x), & \quad  \text{in} \ \mathbb{T}^{n}, 
  \end{cases}
\end{equation}
for $i=1,2,\dots,N.$

Using the second-order variation system for each $u_i(x,t)$, we can address the issue of identifiability for the advection coefficients $\mu_{ij}$, $\nu_{ij}$, the normalization constants $\sum\limits_{j=1}^n\int_{\mathbb{T}^n}\frac{d}{dx_j}\mathbf{k}_{pq,j}(x-y)dy$, and integral kernel ${w_{ij}}$ individually.

In \eqref{highvaryII1} and \eqref{highvaryII2}, $\mathbf{f}_2$ is given arbitrarily no matter what the initial value $\mathbf{u}_0$ is, since the positivity of $\mathbf{u}$ is guaranteed by the positivity of $\mathbf{f}_{1}.$ Moreover, we shall see that the nonlinear terms in the above two systems depend only on the solution of the first order linearization system \eqref{highvaryIsys}. 

\begin{rmk}
    Our method is more versatile than the method of high-order linearization used in \cite{liu2023ipmfg}, where an input of the form \[\mathbf{g}(x;\varepsilon)=\sum_{j=1}^m\varepsilon_j \mathbf{g}_j(x)\] is considered. Indeed, firstly, our method ensures the non-negativity of the solutions by taking $\mathbf{f}_1\geq0$ if $\mathbf{u}_0=\mathbf{0}$. Secondly, our method only gives rise to one system of equations at each order of linearization, which is much simpler than the $\frac{m!}{(m-k)!k!}$ equations at the $k$-th order of linearization for the high-order linearization method.
\end{rmk}

{\centering \section{RECOVERY OF DIFFUSION COEFFICIENTS $\mathbf{d}$ in THEOREM \ref{mainthm_DR}}  \label{Proof_d} }
In this section, we present the proof of Theorem \ref{mainthm_DR}. We will apply a transformative asymptotic technique for the recovery of $\mathbf{d}$. This approach is novel and pioneering, and, to the best of our knowledge, our work is the first research that attempts to recover the diffusion coefficient for parabolic systems with periodic boundary conditions.

We begin by recalling the high order variation method from the previous section, and applying it to our models. We first observe that $\mathbf{u}=\mathbf{0}$ a trivial solution to \eqref{adoptedmodel1} and \eqref{adoptedmodel2}. Therefore, we have the following expansion of $\textbf{u} (x,t;\varepsilon)$ for $l=1,2$:
\begin{equation}\label{general_multiva}
\mathbf{u}^{l}(x,t;\varepsilon)=\varepsilon \mathbf{u}^{l(I)}(x,t)+\frac{1}{2}\varepsilon^2 \mathbf{u}^{l(II)}(x,t)+\tilde{\mathbf{u}}^l(x,t;\varepsilon),
\end{equation}
where we have omitted the higher-order terms for simplicity. Here, $\tilde{\mathbf{u}}^l(x,t;\epsilon)$ satisfies 
\[\frac{1}{|\varepsilon|^3}\norm{\tilde{\mathbf{u}}^l(x,t;\epsilon)}_{[C^{1+\frac{\alpha}{2},2+\alpha}(Q)]^N}=\frac{1}{|\varepsilon|^3}\norm{\mathbf{u}^l(x,t;\varepsilon)-\varepsilon \mathbf{u}^{l(I)}(x,t)-\frac{1}{2}\varepsilon^2 \mathbf{u}^{l(II)}(x,t)}_{[C^{1+\frac{\alpha}{2},2+\alpha}(Q)]^N}\to0
\]
uniformly in $\varepsilon$. By asking $\mathbf{f}_1 \geq 0,$ the systems \eqref{adoptedmodel1} and \eqref{adoptedmodel2} are reduced to
\begin{equation}\label{sec4model1}
\begin{cases}
    \partial_{t}u^{l}_{i}=d^{l}_{i}\Delta u^{l}_{i} + \nabla \cdot \left( u^{l}_i \sum\limits^{N}_{j=1}\mu^{l}_{ij}(\mathbf{k}^{l}_{ij}\ast u^{l}_j)\right),  &\  \text{in}\  Q,\\ 
    u^{l}_i(x,0)=f^{l}_i(x), &\  \text{in}\  \mathbb{T}^{n},\\
\end{cases}
\end{equation}

\begin{equation}\label{sec4model2}
\begin{cases}
    \partial_{t} u^{l}_{i}=d^{l}_{i}\Delta u^{l}_{i} + \nabla \cdot \left( u^{l}_i \sum\limits^{N}_{j=1} \nu^{l}_{ij}\nabla (w^{l}_{ij}\ast u^{l}_j) \right),   &\  \text{in}\  Q,\\ 
    u^{l}_i(x,0)=f^{l}_i(x), &\  \text{in}\  \mathbb{T}^{n},\\
\end{cases}
\end{equation}
with
\[\mathbf{k}^{l}_{ij}\ast u^{l}_j=\int_{\mathbb{T}^{n}} \mathbf{k}^{l}_{ij}(x-y) u^{l}_j(y)dy,\]
\[w^{l}_{ij}\ast u^{l}_j(x)=\int_{\mathbb{T}^{n}} w^{l}_{ij}(x-y) u^{l}_j(y)dy,\]
respectively for $l=1,2,$ and $i=1,2,\dots,N$. Through this setting, we can always ensure the non-negativity for $\mathbf{u}$.

We implement the first-order multi-variation to \eqref{sec4model1} and obtain an initial value problem of the type \eqref{highvaryIsys} for each $i \in \{ 1,\dots,N\}.$ For simplicity, we omit the subscripts $i$ denoting each species, since the proof processes are similar. From the first order linearization, we obtain
\begin{equation}\label{revaI}
\begin{cases} 
\partial_{t}u^{(I)}-d\Delta u^{(I)}=0, &\quad \text{in} \  Q,\\ 
u^{(I)}(x,0)= f_{1}(x), & \quad  \text{in} \ \mathbb{T}^{n}. 
\end{cases} 
\end{equation}

To recover the unknown coefficients $d$, we transform \eqref{revaI} into an equivalent system of unique solvable linear integral equation for it. We will prove the uniqueness of $d$ for $n=2$.

For $p \in \mathbb{C}^{+},$ we denote $\tilde{u}(x,p)= \int^{\infty}_{0}u(x,t)e^{-pt}dt$ as the Laplace transform of $u(x,t)$ for any $p$. Then \eqref{revaI} implies that
\begin{equation}\label{Laplace_varI}
\left(\frac{p}{d(x)}-p\right) \tilde{u}(x,p)+p\tilde{u}(x,p)-\Delta \tilde{u}(x,p)=f_{1}(x) , \quad \text{in} \  Q.
\end{equation}

It is known that the Green's function for the Laplace operator $-\Delta+p$ is given, in $\mathbb{R}^2$, by
\begin{equation}\label{kernelG}
G(x,r,p)=\frac{1}{2\pi} K_0(\sqrt{p}|x-r|),\quad |\arg p|<\pi,
\end{equation}
where $K_0(z)$ is the Macdonald function, which has the representation 
\[K_0(\sqrt{p}|x-r|) = \frac{1}{2}\ln p + \gamma + \ln \frac{|x-r|}{2} + \frac{|x-r|^2}{8}p\ln p + \frac{|x-r|^2}{4}p\ln\frac{|x-r|}{2} + \frac{(\gamma-1)|x-r|^2}{4}p\]
\begin{equation}\label{K0rep}K_0(z) = -\left(\ln\frac{z}{2}+\frac{z^2}{4}\ln\frac{z}{2}+\gamma+\frac{(\gamma-1)z^2}{4}\right)+\mathcal{O}(|z|^3)\end{equation}
as $z\to0$, where $\gamma=0.5772\dots$ is the Euler-Mascheroni constant. Then, we can rewrite \eqref{Laplace_varI} into:
\begin{equation}\label{Inte1}
\tilde{u}(x,p) = \int_{\mathbb{T}^2}G(x,r,p) f_{1}(r)\,dr - p \int_{\mathbb{T}^2} G(x,r,p) \tilde{u}(r,p)\left(\frac{1}{d(r)}-1\right)\,dr.
\end{equation}

Let $S_\rho$ be a circle such that $\overline{\mathbb{T}^2}\subset S_\rho$. Then 
\begin{equation}\label{Inte2}
\tilde{u}(x,p) + p \int_{S_\rho} G(x,r,p) \tilde{u}(r,p)\left(\frac{1}{d(r)}-1\right)\,dr = \int_{S_\rho}G(x,r,p) f_{i,1}(r)\,dr.
\end{equation}

Since $p$ is arbitrary, we can take $p \in S^{+}_{\varepsilon}$ where  $S^{+}_{\varepsilon}=\{p\in \mathbb{C}: \vert p \vert \leq \varepsilon, \text{Re } p >0\}$ for any $\varepsilon<1.$ Then we divide both sides of \eqref{Inte2} by $\ln p$. Denoting $\tilde{v}(x,p)=(\ln p )^{-1} \tilde{u}(x,p),$ we obtain the following equation:
\begin{equation}\label{Inte3}
\tilde{v}(x,p)+ p \int_{S_\rho} G(x,r,p) \tilde{v}(r,p)\left(\frac{1}{d(r)}-1\right)\,dr = \frac{1}{\ln p}\int_{S_\rho}G(x,r,p) f_{1}(r)\,dr.
\end{equation}

We introduce the operator $B_p:L^2(S_{\rho})\to L^2(S_{\rho}),$
\begin{equation}\label{operatorB}
(B_p w)(x)=\int_{S_\rho} pG(x,r,p) \left(1-\frac{1}{d(r)}\right)w(r)\,dr,
\end{equation}
where $x \in S_{\rho}, p \in S^{+}_{\varepsilon}.$ By the continuity of $d$, we know that $B_p$ is completely continuous for all $ p \in S^{+}_{\varepsilon}$ and $\Vert B_p\Vert= O(|p||\ln p|)$ for all small enough $\varepsilon.$ Moreover, $\Vert B_p\Vert<1.$ Viewing the left hand side of \eqref{Inte2} as $(I-B_p)\tilde{u}$, we can obtain
\begin{equation}\label{expansion_v}
\begin{split}
   \tilde{v}(x,p)& =(I-B_p)^{-1}g(x,p)=(I+B_p+B_p^2+\cdots) g(x,p),
\end{split}
\end{equation}
where 
\begin{align*}g(x,p)&= -\frac{1}{2\pi}\int_{S_\rho}\bigg[\frac{1}{2} + \frac{\gamma}{\ln p} + \frac{1}{\ln p}\ln \frac{|x-r|}{2} + \frac{|x-r|^2}{8}p + \frac{|x-r|^2}{4}\frac{p}{\ln p}\ln\frac{|x-r|}{2} \\&\qquad\qquad\qquad+ \frac{(\gamma-1)|x-r|^2}{4}\frac{p}{\ln p}\bigg] f_{1}(r)\,dr+\mathcal{O}(|p|^{3/2}\ln^{-1}p)\\
&=:g_0(x,p)+\mathcal{O}(|p|^{3/2}\ln^{-1}p)\end{align*} using the representation \eqref{K0rep}.

Therefore, we can represent 
\begin{align*}
\tilde{v}(x,p)&=g_0(x,p) + \frac{p\ln p}{4\pi^2}\left(\frac{1}{2} + \frac{\gamma}{\ln p}\right)^2\int_{\mathbb{T}^2} \left(1-\frac{1}{d(r)}\right)\,dr\int_{\mathbb{T}^2} f_{1}(s)\,ds 
\\&\quad +\frac{p}{4\pi^2}\left(\frac{1}{2} + \frac{\gamma}{\ln p}\right)\int_{\mathbb{T}^2}\ln \frac{|x-r|}{2}\left(1-\frac{1}{d(r)}\right)\,dr\int_{\mathbb{T}^2} f_{1}(s)\,ds 
\\&\quad +\frac{p}{4\pi^2}\left(\frac{1}{2}\ln p + \gamma\right)\frac{1}{\ln p}\int_{\mathbb{T}^2}\int_{\mathbb{T}^2}\left(1-\frac{1}{d(r)}\right)\ln \frac{|r-s|}{2} f_{1}(s)\,dr\,ds
\\&\quad +\frac{p}{4\pi^2\ln p}\int_{\mathbb{T}^2}\int_{\mathbb{T}^2}\ln \frac{|x-r|}{2}\left(1-\frac{1}{d(r)}\right)\ln \frac{|r-s|}{2} f_{1}(s)\,dr\,ds
+\mathcal{O}(|p|^{3/2}\ln^{-1}p)\end{align*} 

Assuming that $\epsilon<\exp(-2\gamma)$, we can divide both sides by $\frac{p\ln p}{4\pi^2}\left(\frac{1}{2} + \frac{\gamma}{\ln p}\right)^2$, and conduct a power series expansion in terms of the small parameter $(\ln p + 2\gamma)^{-1}$ to obtain 
\begin{equation}\label{h_expansion}\begin{split}
h(x,p)&=\frac{\tilde{v}(x,p)-g_0(x,p)}{\frac{p\ln p}{4\pi^2}\left(\frac{1}{2} + \frac{\gamma}{\ln p}\right)^2}
\\&=\int_{\mathbb{T}^2} \left(1-\frac{1}{d(r)}\right)\,dr\int_{\mathbb{T}^2} f_{1}(s)\,ds
\\&\quad+\frac{2}{\ln p +2\gamma}\int_{\mathbb{T}^2}\ln \frac{|x-r|}{2}\left(1-\frac{1}{d(r)}\right)\,dr\int_{\mathbb{T}^2} f_{1}(s)\,ds 
\\&\quad +\frac{2}{\ln p +2\gamma}\int_{\mathbb{T}^2}\int_{\mathbb{T}^2}\left(1-\frac{1}{d(r)}\right)\ln \frac{|r-s|}{2} f_{1}(s)\,dr\,ds
\\&\quad +\frac{4}{(\ln p + 2\gamma)^2}\int_{\mathbb{T}^2}\int_{\mathbb{T}^2}\ln \frac{|x-r|}{2}\left(1-\frac{1}{D(r)}\right)\ln \frac{|r-s|}{2} f_{1}(s)\,dr\,ds
+\mathcal{O}(|p|^{1/2}).\end{split}
\end{equation}

By letting $p\to 0,$ we obtain a linear Fredholm integral equation of the first kind, for determining $d(x)$:
\begin{equation}\label{H2}\begin{split}H_2(x)&=\lim_{p\to0}(\ln p + 2\gamma)^2\left[h(x,p)-H_0(x,p)-\frac{H_1(x,p)}{\ln p+2\gamma}\right]\\&=4\int_{\mathbb{T}^2}\int_{\mathbb{T}^2}\ln \frac{|x-r|}{2}\left(1-\frac{1}{d(r)}\right)\ln \frac{|r-s|}{2} f_{1}(s)\,dr\,ds,
\end{split}\end{equation}
where 
\begin{equation}\label{H0}
    H_0(x,p)=\lim_{p\to0} h(x,p)=\int_{\mathbb{T}^2} \left(1-\frac{1}{d(r)}\right)\,dr\int_{\mathbb{T}^2} f_{1}(s)\,ds
\end{equation}
and
\begin{equation}\label{H1}\begin{split}
    H_1(x,p)&=\lim_{p\to0}(\ln p + 2\gamma)\left[h(x,p)-H_0(x,p)\right]\\
    &=2\int_{\mathbb{T}^2}\int_{\mathbb{T}^2}\left(1-\frac{1}{d(r)}\right)\left[\ln \frac{|x-r|}{2} f_{1}(s)\,ds+\ln \frac{|r-s|}{2} f_{1}(s)\,ds\right] \,dr.
    \end{split}
\end{equation}
Since these limits exist and can be computed from the data without knowing $d(x)$ for given inputs $f_{1}(x)$, we can determine $d(x)$ from \eqref{H2}--\eqref{H1}, from the measurement of $u\in\omega\Subset\mathbb{T}^2$. In the particular case when we can choose $f_{1}(x)$ to be the delta function $\delta(x-q)$ for some $q\in\mathbb{T}^2$, \eqref{H2} becomes 
\begin{equation}H_2(x)=4\int_{\mathbb{T}^2}\ln \frac{|x-r|}{2}\left(1-\frac{1}{d(r)}\right)\ln \frac{|r-q|}{2} \,dr.
\end{equation}
Then, since the terms $\ln\frac{|x-r|}{2}$, $\ln\frac{|x-r|}{2}$ are harmonic functions, by the density of the harmonic functions, we uniquely determine $1-\frac{1}{d(x)}$, and therefore uniquely determine $d(x)$ for $h(x,p)$ known by measurement for $x\in\omega\Subset\mathbb{T}^2$. This applies to every species $i$ for $i=1,\dots,N$.

The remainder of the proof for $\mathbf{d}$ follows by noting that the first-order variation system of \eqref{sec4model2} is identical to \eqref{revaI}. Therefore, we can repeat the above proof for the recovery of $\mathbf{d}$ in \eqref{sec4model2}.
\qed

\begin{rmk}
    It should be noted that one is unable to uniquely determine $d(x)$ from \eqref{H2}--\eqref{H1} for general $f_{1}(x)$. For more details, refer to Theorem 2 of \cite{K}.
\end{rmk}

\begin{rmk}
    Note that it is possible to add and subtract terms of the form $p^k\tilde{u}(x,p)$ in \eqref{Laplace_varI} for any $k\geq1$. In particular, one can take $k=2$ to obtain an expansion similar to that of \cite{goncharsky2024}.
\end{rmk} 

\begin{rmk}
    It is important to note that our proof allows for more general parabolic-type operators even with higher order terms, as in \cite{goncharsky2024} for the case of the wave equation. Yet, an intriguing aspect to note is this transformative asymptotic technique only works in the case of $n=2$, since it relies on the asymptotic expansion of the Green's function for $-\Delta+p$ in $\mathbb{R}^2$, so that the operator norm of $B_p$ is small for small $p$. This does not hold in dimensions $n\geq3$, and the inverse problem of recovery of the diffusion coefficient remains open, even in the simple case of the heat equation.
\end{rmk}

Finally, we prove the simple extension in Corollary \ref{cor_DR} where we are able to recover $\mathbf{d}$ in the entire domain $\mathbb{T}^2$. 

\begin{proof}[Proof of Corollary \ref{cor_DR}]
    The corollary follows easily by the analyticity of $\mathbf{d}^l$, by extending the result from the uniqueness of $\mathbf{d}$ in open compact subsets $\omega\Subset\mathbb{T}^2$.
\end{proof}

{\centering \section{RECOVERY OF THE ADVECTION COEFFICIENT $\pmb{\mu}$ AND THE NORMALIZATION CONSTANT OF $\mathbf{k}$ IN THEOREM \ref{mainthm_Inte1} AND COROLLARY \ref{appthm_case1}}  \label{Proof_muk} }
In this section, we recover the advection coefficient $\pmb{\mu}$ and the normalization constant of $\mathbf{k}$ in \eqref{adoptedmodel1}, using the higher-order variations. We will consider the general case of $n$-dimensions. 

\subsection{Unique recovery for the Theorem \ref{mainthm_Inte1}}\label{integral_kRE}
First, we consider the first-order variation system for $l=1,2$ and $i=1,\dots,N$.
\begin{equation}\label{1storder}
\begin{cases} 
\partial_{t}u^{l(I)}_i-d_i\Delta u^{l(I)}_i=0, &\quad \text{in} \  Q,\\ 
u^{l(I)}_i(x,0)= f^{l}_{i,1}(x), & \quad  \text{in} \ \mathbb{T}^{n}. 
\end{cases} 
\end{equation}

Let $\bar{u}^{(I)}_{i}=u^{1(I)}_i-u^{2(I)}_i.$ When $\mathcal{M}^{+}_{\pmb{\mu}^1,\mathbf{k}^1}(\mathbf{f}^1)=\mathcal{M}^{+}_{\pmb{\mu}^2, \mathbf{k}^2}(\mathbf{f}^2)$, we obtain
\begin{equation}\label{unqiue_variI}
\begin{cases} 
\partial_{t}\bar{u}^{(I)}_{i}-d_i\Delta\bar{u}^{(I)}_{i}=0, &\quad \text{in} \  Q,\\ 
\bar{u}^{(I)}_{i}(x,0)= 0, & \quad  \text{in} \ \mathbb{T}^{n}. 
\end{cases} 
\end{equation}
By the uniqueness of solution of the heat equation, \eqref{unqiue_variI} only has the trivial solution. Thus we have 
\begin{equation}\label{1st_multivari}
\textbf{u}^{1(I)}(x,t)=\textbf{u}^{2(I)}(x,t)=:\textbf{u}^{(I)}(x,t).\end{equation}

Next, we give the second-order variation system for \eqref{sec4model1}:
\begin{equation}\label{var2k_general}
    \begin{cases}
    \partial_{t} u^{l(II)}_{i}= d_{i}\Delta u^{l(II)}_{i}-\nabla \cdot \left[u^{(I)}_{i} \sum\limits^N_{j=1}\mu^{l}_{ij} (\mathbf{k}^{l}_{ij} * u^{(I)}_{j}) \right], & \quad \text{in}\  Q,\\ 
    u^{l(II)}_{i}(x,0)= 2f^{l}_{i,2}(x), & \quad  \text{in} \ \mathbb{T}^{n}. 
  \end{cases}
\end{equation}

\textbf{\textit{Recovery of the advection coefficient $\pmb{\mu}$} } 
Assume that the normalization constant of the integral kernel $\mathbf{k}$ is given and non-zero. To begin, we first recover the advection coefficients ${\mu_{11}, \mu_{22},\dots, \mu_{NN}}$. Using different inputs of $\mathbf{f}_1$ for each coefficient $\mu_{ii}$, we select a suitable $\mathbf{u}^{(I)}$ such that $\mathbf{u}^{(I)}(x,t)=(0,\dots,1,\dots,0)$ is non-trivial only for the $i$-th species. Here, we take the scenario where $i=1$ as an example, and the remaining $\mu_{ii}$ $(i=2,\dots, N)$ can be identified using a similar approach.

Let $\hat{u}_1^{(II)}=u_1^{1(II)}-u_1^{2(II)},$ $\hat{\mu}_{11}=\mu_{11}^{1}-\mu_{11}^{2}$. Taking $\mathbf{u}^{(I)}(x,t)=(1,0,\dots,0)$, we obtain from \eqref{var2k_general} that $\hat{u}_{1}^{(II)}$ satisfies
  \begin{equation}\label{varmu_hat}
    \begin{cases}
    \partial_{t} \hat{u}^{(II)}_{1}= d_{1}\Delta \hat{u}^{(II)}_{1}-\nabla \cdot \left[ \hat{\mu}_{11} \int_{\mathbb{T}^n}\mathbf{k}_{11}(x-y)dy \right], & \quad \text{in}\  Q,\\ 
    \hat{u}^{(II)}_{1}(x,0)= 0, & \quad  \text{in} \ \mathbb{T}^{n}, 
  \end{cases}
  \end{equation}
  if $\mathcal{M}^{+}_{\pmb{\mu}^1,\mathbf{k}^1}(\mathbf{f}^1)=\mathcal{M}^{+}_{\pmb{\mu}^2, \mathbf{k}^2}(\mathbf{f}^2)$. 
One may see the term $\sum\limits^N_{j=1}\mu^{l}_{1j} (\mathbf{k}_{1j} \ast u^{(I)}_{j})$ only leaves $\mu^{l}_{11}(\mathbf{k}_{11}\ast u^{(I)}_{1})$ as $u^{(I)}_{j}=0$ for $j=2,\dots,N$. Substituting $u^{(I)}_{1}=1$ to the convolution gives the equation in \eqref{varmu_hat}.

Let $\omega_{1}$ be the solution of $-\partial_{t} \omega_{1}-d_1 \Delta \omega_{1}=0$ in $Q$, multiply it on both sides of \eqref{varmu_hat} and apply integration by parts, yielding:
\begin{equation}\label{omegai_formu}
   \int_{Q} -\nabla \cdot \left[ \hat{\mu}_{11} \int_{\mathbb{T}^n} \mathbf{k}_{11}(x-y)dy \right] \omega_1 dxdt=0.
\end{equation}

Substituting the CGO solution of $\omega_1=e^{-|\xi_1|^2 t- \frac{i}{\sqrt{d_1}}\xi_1 \cdot x}$ into \eqref{omegai_formu}, we have
\begin{equation}\label{omegaiseparable_mu}
    \int_{\mathbb{T}^n} -\nabla \cdot \left[  \hat{\mu}_{11}  \int_{\mathbb{T}^n} \mathbf{k}_{11}(x-y)dy \right] e^{- \frac{i}{\sqrt{d_1}}\xi_1 \cdot x} dx=0.
\end{equation}

By viewing the left hand side of \eqref{omegaiseparable_mu} as the Fourier transform in $x$, to ensure \eqref{omegaiseparable_mu}, $\nabla \cdot [  \hat{\mu}_{11}  \int_{\mathbb{T}^n} \mathbf{k}_{11}(x-y)dy ]$ must be zero, which implies that 
\begin{equation}\label{unique_mu}
    \sum\limits_{j=1}^n\hat{\mu}_{11} \int_{\mathbb{T}^n}\frac{d}{dx_j}k_{11,j}(x-y)dy=0.
\end{equation}
Since the normalization constant of $\mathbf{k}_{11}$ is assumed to be non-zero, we have $\hat{\mu}_{11}=0$, that is
\begin{equation}\label{mu11eq}\mu^{1}_{11}=\mu^{2}_{11}.\end{equation}

Next, we recover the remaining advection coefficients $\mu_{ij}$ for $i\neq j$. We will show the case when $i=1$ and $j=2$, and the other cases follow similarly. This time, we choose $\mathbf{f}_1$ such that  $\mathbf{u}^{(I)}(x,t)=(1,1,0,\dots,0)$, which is non-trivial only at the first and second positions. With this configuration, the term $\sum\limits^{N}_{j=1}\mu^{l}_{1j} (\mathbf{k}_{1j} \ast u^{(I)}_{j})$ now simplifies to $\mu_{11}(\mathbf{k}_{11}\ast u^{(I)}_{1})+\mu^{l}_{12}(\mathbf{k}_{12}\ast u^{(I)}_{2})$. By substituting $u^{(I)}_{1}=1$ and $u^{(I)}_{2}=1$ into the convolution, we obtain:
\begin{equation}\label{varmu2_hat}
    \begin{cases}
    \partial_{t} \hat{u}^{(II)}_{1}= d_{1}\Delta \hat{u}^{(II)}_{1}-\nabla \cdot [ \hat{\mu}_{11} \int_{\mathbb{T}^n}\mathbf{k}_{11}(x-y)dy +\hat{\mu}_{12} \int_{\mathbb{T}^n}\mathbf{k}_{12}(x-y)dy], & \quad \text{in}\  Q,\\ 
    \hat{u}^{(II)}_{1}(x,0)= 0, & \quad  \text{in} \ \mathbb{T}^{n}, 
  \end{cases}
  \end{equation}
  where $\hat{u}_1^{(II)}=u_1^{1(II)}-u_1^{2(II)}$ once again, while $\hat{\mu}_{ij}=\mu_{ij}^{1}-\mu_{ij}^{2}$ for $i=1$, $j=1,2$, if $\mathcal{M}^{+}_{\pmb{\mu}^1,\mathbf{k}^1}(\mathbf{f}^1)=\mathcal{M}^{+}_{\pmb{\mu}^2, \mathbf{k}^2}(\mathbf{f}^2)$.

Recall that we have already obtained the conclusion $\hat{\mu}_{11}=0$ from the previous step in \eqref{mu11eq}. To proceed, we multiply the same $\omega_1$ from the previous step on both sides of \eqref{varmu2_hat} and apply integration by parts, and similarly get 
\begin{equation}\label{omega2_mu}
   \int_{Q} -\nabla \cdot \left[ \hat{\mu}_{12} \int_{\mathbb{T}^n} \mathbf{k}_{12}(x-y)dy \right] \omega_1 dxdt=0.
\end{equation}

By substituting the CGO solution of $\omega_1$ to \eqref{omega2_mu} and separating variables, we derive the following equation:
\begin{equation}\label{omegasp2_mu}
    \int_{\mathbb{T}^n} -\nabla \cdot \left[  \hat{\mu}_{12}  \int_{\mathbb{T}^n} \mathbf{k}_{12}(x-y)dy \right] e^{- \frac{i}{\sqrt{d_1}}\xi_1 \cdot x} dx=0.
\end{equation}

Once again, identifying the left hand side \eqref{omegasp2_mu} with the Fourier transform in $x$, \eqref{omegasp2_mu} is equivalent to $\nabla \cdot [  \hat{\mu}_{12}  \int_{\mathbb{T}^n} \mathbf{k}_{12}(x-y)dy ] =0$. Since the normalization constant of $\mathbf{k}_{12}$ is non-zero, we can conclude that  $\hat{\mu}_{12}=0$, which is
\[\mu^{1}_{12}=\mu^{2}_{12}.\]

By considering $\mathbf{u}^{(I)}(x,t)=(1,0,\dots, 1,\dots,0)$, which is non-trivial only at the first and $j$-th $(j=3,4,\dots,N)$ positions, we can repeat the aforementioned procedure and identify $\mu_{1j}$ $(j=3,4,\dots,N)$. Similarly, by choosing the initial value of $\mathbf{u}^{(I)}$ such that $\mathbf{u}^{(I)}(x,t)=(0,\dots,1,\dots,0)$, which is non-trivial only at the $i$-th position, we can recover $\mu_{ii}$ by repeating the first part of this proof, with a different CGO solution of $(-\partial_t-d_i\Delta)\omega_i=0$. Additionally, by modifying $\mathbf{u}^{(I)}(x,t)=(0,\dots,1,\dots,1,\dots,0)$ to be non-trivial at the $i$-th $(i=2,3,\dots,N)$ and the $j$-th $(j\neq i)$ positions, we can recover $\mu_{ij}$ (including $\mu_{21}$ when $i=2$ and $j=1$ which is not necessarily equal to $\mu_{12}$) making use of the CGO solution of $\omega_i$. By repeating this process, all the advection coefficients $\mu_{ij}$ $(i,j=1,\dots,N)$ can be uniquely identified. This leads us to the conclusion
\begin{equation}\label{conclusion_mu}
    \pmb{\mu}^{1}= \pmb{\mu}^{2}
\end{equation}
when $\mathcal{M}^{+}_{\pmb{\mu}^1,\mathbf{k}^1}(\mathbf{f}^1)=\mathcal{M}^{+}_{\pmb{\mu}^2, \mathbf{k}^2}(\mathbf{f}^2)$.
\qed

\textbf{\textit{Recovery of the normalization constants $\mathbf{k}$} }
In this part, we consider the case where all the advection coefficients $\pmb{\mu}$ are known and non-zero. We begin by focusing on the recovery of $\mathbf{k}^{l}_{11}$ $(l=1,2)$. By carefully selecting the input initial value function, we can choose $\mathbf{u}^{(I)}$ as $(1,0,\dots,0)$, which is non-trivial only at the first position. Consequently, \eqref{var2k_general} can be expressed as:
\begin{equation}\label{var2k_k11}
    \begin{cases}
    \partial_{t} u^{l(II)}_{1}= d_{1}\Delta u^{l(II)}_{1}-\nabla \cdot [ \mu_{11} \int_{\mathbb{T}^n} \mathbf{k}^{l}_{11}(x-y)dy ], & \quad \text{in}\  Q,\\ 
    u^{l(II)}_{1}(x,0)= 2f^{l}_{1,2}(x), & \quad  \text{in} \ \mathbb{T}^{n}. 
  \end{cases}
\end{equation}

Let $\bar{\mathbf{u}}^{(II)}=\mathbf{u}^{1(II)}-\mathbf{u}^{2(II)},$ $\bar{\mathbf{k}}_{11}=\mathbf{k}^{1}_{11}-\mathbf{k}^{2}_{11}.$ Meanwhile, from the given condition $\mathcal{M}^{+}_{\pmb{\mu}^1,\mathbf{k}^1}=\mathcal{M}^{+}_{\pmb{\mu}^2,\mathbf{k}^2}$, we now obtain
\begin{equation}\label{var2k_bark11}
    \begin{cases}
    \partial_{t} \bar{u}^{(II)}_{1}= d_{1}\Delta \bar{u}^{(II)}_{1}-\nabla \cdot [ \mu_{11} \int_{\mathbb{T}^n} \bar{\mathbf{k}}_{11}(x-y)dy ], & \quad \text{in}\  Q,\\ 
    \bar{u}^{(II)}_{1}(x,0)= 0, & \quad  \text{in} \ \mathbb{T}^{n}. 
  \end{cases}
\end{equation}

Once again, suppose $\omega_{1}$ is the solution of $-\partial_{t} \omega_{1}-d_1 \Delta \omega_{1}=0$ in $Q$. Multiply both sides of \eqref{var2k_bark11} by $\omega_{1}$ and perform integration by parts, resulting in:
\begin{equation}\label{omega1_fork11}
    \int_{Q} -\nabla \cdot \left[ \mu_{11} \int_{\mathbb{T}^n} \bar{\mathbf{k}}_{11}(x-y)dy \right] \omega_1 dxdt=0.
\end{equation}

We once again consider the CGO solution $e^{-|\xi_1|^2 t- \frac{i}{\sqrt{d_1}}\xi_1 \cdot x}$ for $\omega_1$. By substituting this solution into \eqref{omega1_fork11}, we can transform it into the following expression:
\begin{equation}\label{omega1separable_k11}
    \int_{\mathbb{T}^n} -\nabla \cdot \left[ \mu_{11} \int_{\mathbb{T}^n} \bar{\mathbf{k}}_{11}(x-y)dy \right] e^{- \frac{i}{\sqrt{d_1}}\xi_1 \cdot x} dx=0.
\end{equation}

Denote $\textbf{(A)}:= \nabla \cdot [ \mu_{11} \int_{\mathbb{T}^n} \bar{\mathbf{k}}_{11}(x-y)dy ]$. To ensure \eqref{omega1separable_k11}, $\textbf{(A)}$ should equal to $0$, which indicates
\begin{equation}\label{unique_k11}
    \sum\limits_{j=1}^n\mu_{11} \int_{\mathbb{T}^n}\frac{d}{dx_j}\bar{\mathbf{k}}_{11,j}(x-y)dy=0.
\end{equation}

The conclusion becomes evident. As the advection coefficient $\mu_{11}$ is known and non-zero, the normalization constant of $\bar{\mathbf{k}}_{11}$ must be zero. Consequently, we arrive at the following result:
\begin{equation}\label{equal_k11}
    \sum\limits_{j=1}^n\int_{\mathbb{T}^n}\frac{d}{dx_j}\mathbf{k}^{1}_{11,j}(x-y)dy=\sum\limits_{j=1}^n\int_{\mathbb{T}^n}\frac{d}{dx_j}\mathbf{k}^{2}_{11,j}(x-y)dy.
\end{equation}

By employing a similar approach, we can also establish the uniqueness of the normalization constants for $\mathbf{k}_{22},\mathbf{k}_{33},\dots,\mathbf{k}_{NN}$ in this way.

Next, by choosing $\mathbf{u}^{(I)}(x,t)=(1,1,0,\dots,0)$, which is non-trivial only on the first and second places, we have the following equation for $\bar{u}^{(II)}_{1}$ and $\bar{\mathbf{k}}_{12}:=\mathbf{k}^{1}_{12}-\mathbf{k}^{2}_{12}$:
\begin{equation}\label{vark2_bar}
    \begin{cases}
    \partial_{t} \bar{u}^{(II)}_{1}= d_{1}\Delta \bar{u}^{(II)}_{1}-\nabla \cdot [ \mu_{11} \int_{\mathbb{T}^n}\bar{\mathbf{k}}_{11}(x-y)dy +\mu_{12} \int_{\mathbb{T}^n}\bar{\mathbf{k}}_{12}(x-y)dy], & \quad \text{in}\  Q,\\ 
    \bar{u}^{(II)}_{1}(x,0)= 0, & \quad  \text{in} \ \mathbb{T}^{n}. 
  \end{cases}
\end{equation}

Once more, we multiply both sides of \eqref{vark2_bar} by the solution $\omega_1$ of $(-\partial_t-d_1\Delta)\omega_1=0$ and perform integration by parts. By substituting the CGO solution of $\omega_1$ and separating variables for the integrals, we obtain the following equation:
\begin{equation}\label{omega1separable_k12}
    \int_{\mathbb{T}^n} -\nabla \cdot \left[ \mu_{11} \int_{\mathbb{T}^n} \bar{\mathbf{k}}_{11}(x-y)dy +\mu_{12} \int_{\mathbb{T}^n} \bar{\mathbf{k}}_{12}(x-y)dy \right] e^{- \frac{i}{\sqrt{d_1}}\xi_1 \cdot x} dx=0.
\end{equation}

Denote $\textbf{(B)}:= \nabla \cdot [ \mu_{11} \int_{\mathbb{T}^n} \bar{\mathbf{k}}_{11}(x-y)dy +\mu_{12} \int_{\mathbb{T}^n} \bar{\mathbf{k}}_{12}(x-y)dy ]$. In order to satisfy \eqref{omega1separable_k12}, the condition $\textbf{(B)} = 0$ must hold, indicating 
\begin{equation}\label{unique_k12}
    \sum\limits_{j=1}^n\mu_{11} \int_{\mathbb{T}^n}\frac{d}{dx_j}\bar{\mathbf{k}}_{11,j}(x-y)dy+    \sum\limits_{j=1}^n\mu_{12} \int_{\mathbb{T}^n}\frac{d}{dx_j}\bar{\mathbf{k}}_{12,j}(x-y)dy=0.
\end{equation}

Since we have already verified that the normalization constant of $\bar{\mathbf{k}}_{11}$ is zero, and $\mu_{11}$ and $\mu_{12}$ are given constants, \eqref{unique_k12} only involves the normalization constant for $\bar{\mathbf{k}}_{12}$. Hence, it is evident that
\begin{equation}\label{equal_k12}
    \sum\limits_{j=1}^n\int_{\mathbb{T}^n}\frac{d}{dx_j}\mathbf{k}^{1}_{12,j}(x-y)dy=\sum\limits_{j=1}^n\int_{\mathbb{T}^n}\frac{d}{dx_j}\mathbf{k}^{2}_{12,j}(x-y)dy.
\end{equation}

Thus, we have successfully identified the normalization constant of $\mathbf{k}_{12}$. Following the recovery conclusion of $\mathbf{k}_{11}$, one can select $\mathbf{u}^{(I)}(x,t)=(1,0,\dots, 1,\dots,0)$, which is non-trivial only at the first and $j$-th $(i=3,4,\dots,N)$ positions, to recover the normalization constant of $\mathbf{k}_{1j}$. Similarly, considering the equation for $u^{(II)}_i$, it becomes evident that we can recover any $\mathbf{k}_{ij}$ $(i=2,\dots,N; j=1,\dots,N)$ once we have recovered $\mathbf{k}_{ii}$ using the first part of this proof for $\mathbf{k}$. Thus, the identifiability process is complete. \qed

\subsection{Application to the case of Corollary \ref{appthm_case1}}
 Corollary \ref{appthm_case1} can be directly derived from Theorem \ref{mainthm_DR} and \ref{mainthm_Inte1}. Indeed,

 \begin{proof}[Proof of Corollary \ref{appthm_case1}]

Based on Lemma \ref{wellposed_eco1}, the system \eqref{incon_lem1} possesses a unique, global, and non-negative solution. Consequently, for any initial function satisfying $\mathbf{(L_1)}$, the condition $\mathcal{M}^{+}_{\mathbf{d}^1}(\mathbf{f})=\mathcal{M}^{+}_{\mathbf{d}^2}(\mathbf{f})$ leads to $\mathbf{d}^1=\mathbf{d}^2$.
Moreover, for any $\mathbf{f}^l$ that satisfies $\mathbf{(L_1)}$ and has a known non-zero normalization constant for $\mathbf{K}$, the conditions specified in Corollary \ref{appthm_case1} align with the prerequisites of Theorem \ref{mainthm_Inte1}. Then if $\mathcal{M}^{+}_{\pmb{\mu}^1,\mathbf{K}^1}(\mathbf{f}^1)=\mathcal{M}^{+}_{\pmb{\mu}^2,\mathbf{K}^2}(\mathbf{f}^2)$, we can conclude that $\pmb{\mu}^1=\pmb{\mu}^2$. Conversely, if we know the non-zero advection coefficient $\pmb{\mu}$ and $\mathcal{M}^{+}_{\pmb{\mu}^1,\mathbf{K}^1}(\mathbf{f}^1)=\mathcal{M}^{+}_{\pmb{\mu}^2,\mathbf{K}^2}(\mathbf{f}^2)$, then the normalization constants for $\mathbf{K}^1$ and $\mathbf{K}^2$ are identical. 
 \end{proof}

{\centering \section{RECOVERY OF THE ADVECTION COEFFICIENT $\pmb{\nu}$ AND THE INTEGRAL KERNEL $\mathbf{w}$ IN THEOREM \ref{mainthm_Inte2} AND COROLLARY \ref{appthm_case2}}  \label{Proof_nuw} }
In this section, we will give the recovery of the advection coefficient $\pmb{\nu}$ and the integral kernel $\mathbf{w}$. The proofs are similar to those in Section \ref{Proof_muk}. We also verify the proof in the space $\mathbb{T}^n$. Before we begin, we give a crucial auxiliary lemma.

\begin{lem}\label{formu_interp}
Consider 
\begin{equation}\label{lem1_interp} 
\partial_t \mathbf{u}(x,t)-\mathbf{p}\Delta \mathbf{u}(x,t)= 0, \quad \text{in} \   Q,
\end{equation}
where $\mathbf{p}=(p_1,\dots,p_N)$ for constants $p_i$, $i=1,\dots,N$. Then there exists a sequence of periodic solutions $u_i(x,t)$ to the system \eqref{lem1_interp} such that $u_i(x,t)=e^{\lambda_i t} U_i(x;\lambda_i)$ for some non-trivial $\lambda_i \in \mathbb{R}^n$ and $U_i(x;\lambda_i) \in C^{2}(\mathbb{T}^{n}).$ In particular, $U_i(x;\lambda_i)=e^{i\xi\cdot x}$ with $|\xi|^2=\frac{\lambda_i}{p_i}$ is not necessarily a trivial function, and $\frac{\lambda_i}{p_i}$ is its corresponding eigenvalue. Furthermore, there does not exist an open subset $T$ of $\mathbb{T}^{n}$ such that $\nabla U_i(x;\lambda_i)=0$ in $T.$
\end{lem}

\begin{proof} Let $\frac{\lambda_i}{p_i}$ be an eigenvalue of the periodic Laplacian  and $U_i(x;\lambda_i)$ be its eigenfunction. Then, $U_i(x;\lambda_i)$ satisfies the equation 
\begin{equation}\label{lem1p_interp} 
-\Delta U_i(x;\lambda_i)= \frac{\lambda_i}{p_i}U_i(x;\lambda_i),  \quad  \text{in}\   Q,
\end{equation}
It is obvious to see that $ u_i(x,t)=e^{\lambda_i t} U_i(x;\lambda_i) $ is a solution of \eqref{lem1_interp}. Furthermore, if we suppose that there is an open subset $T$ of $\mathbb{T}^{n}$ such that $\nabla U_i(x;\lambda_i)=0$ in $T$, then $U_i(x,\lambda_i)$ is a constant in $T.$ Since $U_i(x,\lambda_i)$ is any eigenfunction of $\Delta,$ we have a contradiction.

The proof is complete. 
\end{proof}

\subsection{Unique recovery for the Theorem \ref{mainthm_Inte2}}
With this lemma in hand, we can now proceed with the proof of  Theorem \ref{mainthm_Inte2}. We observe that the 
first order variation of \eqref{sec4model2} is in the form \eqref{1storder}, and therefore, as in Section \ref{integral_kRE}, $\textbf{u}^{1(I)}(x,t)=\textbf{u}^{2(I)}(x,t)=:\textbf{u}^{(I)}(x,t)$ when $\mathcal{M}^{+}_{\pmb{\nu}^1,\mathbf{w}^1}(\mathbf{f}^1)=\mathcal{M}^{+}_{\pmb{\nu}^2 \mathbf{w}^2}(\mathbf{f}^2)$. By assuming each element of $\pmb{d}=(d_1,d_2,\dots,d_{N})$ is a constant, it is apparent that the expression of $\textbf{u}^{(I)}$ in \eqref{unqiue_variI} satisfies an equation of the form \eqref{lem1_interp}. Therefore, based on Lemma \ref{formu_interp}, there exist
\[\pmb{\lambda}=(\lambda_{1},\lambda_{2},\dots,\lambda_{N}) \in \mathbb{R}^{n}, \quad \textbf{U}(x)=\big( U_1(x), U_2(x), \dots, U_N(x) \big) \in C^2(\mathbb{T}^{n}) \]
such that $e^{\pmb{\lambda} t}\textbf{U}(x)$ satisfies equations for $\textbf{u}^{(I)}(x,t)$ in \eqref{1storder}. And each $u^{(I)}_i$ satisfies:
\begin{equation}\label{CGO_uI}
    u^{(I)}_i(x,t)= e^{\lambda_{i} t}U_{i}(x).
\end{equation}

 In order to recover the advection coefficient $\pmb{\nu}$ and the integral kernel $\mathbf{w}$, it is necessary to consider the following second-order variation system for \eqref{sec4model2}:
\begin{equation}\label{var2w_general}
    \begin{cases}
    \partial_{t} u^{l(II)}_{i}= d_{i}\Delta u^{l(II)}_{i}-\nabla \cdot \left[u^{(I)}_{i} \sum\limits^{N}_{j=1}\nu^{l}_{ij} \nabla (w^{l}_{ij} \ast u^{(I)}_{j}) \right], & \quad \text{in}\  Q,\\ 
    u^{l(II)}_{i}(x,0)= 2f^{l}_{i,2}(x), & \quad  \text{in} \ \mathbb{T}^{n}. 
  \end{cases}
\end{equation}

\textbf{\textit{Recovery of the advection coefficient $\pmb{\nu}$} } 
The recovery process begins with the advection coefficients on the diagonal $\{\nu_{11}, \nu_{22},\dots, \nu_{NN}\}$. Assuming that the integral kernel $\mathbf{w}$ is a known non-trivial function, we can manipulate the input $\mathbf{f}_1$ such that $\mathbf{u}^{(I)}(x,t)=(e^{\lambda_{1} t}U_{1}(x),0,\dots,0)$, where $\lambda_1$ is a non-zero constant and $U_{1}(x)$ satisfies the Lemma \ref{formu_interp}. This choice results in the term  $\sum\limits^{N}_{j=1}\nu^{l}_{1j} \nabla(w_{1j} \ast u^{(I)}_{j})$ simplifying to $\nu^{l}_{11} \nabla(w_{11}\ast u^{(I)}_{1})$ as $u^{(I)}_{j}=0$ for $j=2,\dots,N$.  Consequently, \eqref{var2w_general} can be expressed as
\begin{equation}\label{var2_nu11}
    \begin{cases}
    \partial_{t} u^{l(II)}_{1}= d_{1}\Delta u^{l(II)}_{1}-\nabla \cdot [ \nu^{l}_{11} u^{(I)}_{1} \nabla  (w_{11}\ast u^{(I)}_{1})], & \quad \text{in}\  Q,\\ 
    u^{l(II)}_{1}(x,0)= 2f^{l}_{1,2}(x), & \quad  \text{in} \ \mathbb{T}^{n}. 
  \end{cases}
\end{equation}

Let $\hat{u}_1^{(II)}=u_1^{1(II)}-u_1^{2(II)},$ $\hat{\nu}_{11}=\nu_{11}^{1}-\nu_{11}^{2}$. Taking the difference of the equations for $l=1,2$ of \eqref{var2_nu11}, we obtain
  \begin{equation}\label{varnu_hat}
    \begin{cases}
    \partial_{t} \hat{u}^{(II)}_{1}= d_{1}\Delta \hat{u}^{(II)}_{1}-\nabla \cdot [ \hat{\nu}_{11} u^{(I)}_{1} \nabla  (w_{11}\ast u^{(I)}_{1})], & \quad \text{in}\  Q,\\ 
    \hat{u}^{(II)}_{1}(x,0)= 0, & \quad  \text{in} \ \mathbb{T}^{n}. 
  \end{cases}
  \end{equation}

Due to the property of convolution, we have 
\[\nabla (w_{11}\ast u^{(I)}_{1})= w_{11}\ast (\nabla u^{(I)}_{1}).\]

Let $\omega_{1}$ be the solution of $-\partial_{t} \omega_{1}-d_1 \Delta \omega_{1}=0$ in $Q$, if we multiply it on both sides of \eqref{varnu_hat} and perform integration by parts, we obtain
\begin{equation}\label{omega1_fornu11}
    \int_{Q} -\nabla \cdot [\hat{\nu}_{11} u^{(I)}_{1} (w_{11}\ast \nabla u^{(I)}_{1})]\omega_1 dxdt=0.
\end{equation}

By substituting the CGO solution $\omega_1=e^{-|\xi_1|^2 t- \frac{i}{\sqrt{d_1}}\xi_1 \cdot x}$ and the specific expression for $U_1(x)=e^{i \xi \cdot x}$, where $\vert \xi \vert^2=\frac{\lambda_1}{d_1}$ as given in Lemma \ref{lem1_interp}, into \eqref{omega1_fornu11} and then separating variables, we obtain
\begin{equation}\label{omega1separable_nu11}
    \int_{\mathbb{T}^n} \int_{\mathbb{T}^n}  \nabla \cdot  [\hat{\nu}_{11} e^{i \xi \cdot x} (w_{11}(y)i\xi e^{i \xi \cdot (x-y)})]  e^{- \frac{i}{\sqrt{d_1}}\xi_1 \cdot x} dydx=0.
\end{equation}
Expanding, this gives
\begin{align*}
    0&=\int_{\mathbb{T}^n} \int_{\mathbb{T}^n}  \nabla \cdot  [\hat{\nu}_{11} e^{i \xi \cdot x} (w_{11}(y)i\xi e^{i \xi \cdot (x-y)})]  e^{- \frac{i}{\sqrt{d_1}}\xi_1 \cdot x} dydx\\
    &=\int_{\mathbb{T}^n} \int_{\mathbb{T}^n} \sum_{j=1}^n (2i\xi_j)\cdot(i\xi_j) \hat{\nu}_{11} w_{11}(y) e^{i \xi \cdot(2x-y)}  e^{- \frac{i}{\sqrt{d_1}}\xi_1 \cdot x} dydx\\
    &=-\frac{2\lambda_1}{d_1}\int_{\mathbb{T}^n} \int_{\mathbb{T}^n} \hat{\nu}_{11} w_{11}(y) e^{i \xi \cdot(2x-y)}  e^{- \frac{i}{\sqrt{d_1}}\xi_1 \cdot x} dydx.
\end{align*}

Separating out the integral in $x$, \eqref{omega1separable_nu11} is equal to:
\begin{equation}\label{decide_nu11}
    \frac{2\lambda_1}{d_1} \hat{\nu}_{11} \int_{\mathbb{T}^n} e^{i \xi \cdot (-y)}w_{11}(y)dy =0.
\end{equation}

 Since we know that $\lambda_1$ is non-zero and $w_{11}$ is a known non-trivial function, \eqref{decide_nu11} implies that $\hat{\nu}_{11}=0$. In other words, we have
\begin{equation}\label{unique_nu11}
    \nu^{1}_{11}= \nu^{2}_{11}.
\end{equation}

 Moving forward, our objective is to determine $\nu_{1j}$ for $j=2,\dots,N$. We select  $\mathbf{f}_1$ such that $\mathbf{u}^{(I)}(x,t)=(e^{\lambda_{1} t}U_{1}(x),\dots,e^{\lambda_{j} t}U_{j}(x),\dots, 0)$, where the non-trivial values are only present in the first and $j$-th positions. In this case, the term $\sum\limits^{N}_{j=1}\nu^{l}_{1j} \nabla(w_{1j} \ast u^{(I)}_{j})$ simplifies to $\nu^{l}_{11} \nabla(w_{11}\ast u^{(I)}_{1})+\nu^{l}_{1j} \nabla(w_{1j}\ast u^{(I)}_{j})$. Consequently, \eqref{var2w_general} for $l=1,2$ can be expressed as
\begin{equation}\label{var2_nu1j}
    \begin{cases}
    \partial_{t} u^{l(II)}_{1}= d_{1}\Delta u^{l(II)}_{1}-\nabla \cdot [ \nu^{l}_{11} u^{(I)}_{1} \nabla  (w_{11}\ast u^{(I)}_{1})+\nu^{l}_{1j} \nabla(w_{1j}\ast u^{(I)}_{j})], & \quad \text{in}\  Q,\\ 
    u^{l(II)}_{1}(x,0)= 2f^{l}_{1,2}(x), & \quad  \text{in} \ \mathbb{T}^{n}. 
  \end{cases}
\end{equation}

And the equation for $\hat{u}^{(II)}_{1}$ is:
  \begin{equation}\label{varnu1j_hat}
    \begin{cases}
    \partial_{t} \hat{u}^{(II)}_{1}= d_{1}\Delta \hat{u}^{(II)}_{1}-\nabla \cdot [ \hat{\nu}_{11} u^{(I)}_{1} \nabla  (w_{11}\ast u^{(I)}_{1})+\hat{\nu}_{1j} u^{(I)}_{j} \nabla  (w_{1j}\ast u^{(I)}_{j})], & \quad \text{in}\  Q,\\ 
    \hat{u}^{(II)}_{1}(x,0)= 0, & \quad  \text{in} \ \mathbb{T}^{n}. 
  \end{cases}
  \end{equation}

Having established that $\hat{\nu}_{11}=0$, we can utilize the CGO solution of $\omega_1$ and the particular solution of $U_j$ to multiply both sides of \eqref{varnu1j_hat} and subsequently employ integration by parts. This leads to
\begin{equation}\label{omega1separable_nu1j}
    \int_{\mathbb{T}^n} \int_{\mathbb{T}^n}  \nabla \cdot  [\hat{\nu}_{1j} e^{i \xi \cdot x} (w_{1j}(y) i \xi e^{i \xi \cdot (x-y)})]  e^{- \frac{i}{\sqrt{d_1}}\xi_1 \cdot x} dydx=0.
\end{equation}
By simplification as done for $\nu_{11}$, this gives
\begin{equation}\label{decide_nu1j}
    \frac{2\lambda_j}{d_j} \hat{\nu}_{1j} \int_{\mathbb{T}^n} e^{i \xi \cdot (-y)}w_{1j}(y)dy =0.
\end{equation}

Considering the non-zero value of $\lambda_j$ and recognizing $w_{1j}$ as a non-trivial function, equation \eqref{decide_nu1j} indicates that $\hat{\nu}_{1j}=0$. In other words, we can conclude that:
\begin{equation}\label{unique_nu1j}
    \nu^{1}_{1j}= \nu^{2}_{1j}, \quad j=2,\dots,N.
\end{equation}

So far, we have successfully determined the advection coefficients for $\nu_{1j}$ (where $j=1,\dots,N$). The remaining advection coefficients $\nu_{ij}$ (where $i=2,\dots,N$ and $j=1,\dots,N$) can be obtained using a similar approach. We begin by establishing the uniqueness for $\nu_{ii}$ by selecting $\mathbf{f}_1$ such that $\mathbf{u}^{(I)}(x,t)=(0,\dots,e^{\lambda_i t}U_i(x),\dots,0)$, which is non-trivial only at the $i$-th position. Then, we modify $\mathbf{f}_1$ such that it becomes non-trivial at both the $i$-th and $j$-th positions, with the corresponding first order solutions being $e^{\lambda_i t}U_i(x)$ and $e^{\lambda_j t}U_j(x)$. Consequently, we arrive at the conclusion:
\begin{equation}\label{conclusion_nu}
\pmb{\nu}^{1}= \pmb{\nu}^{2}.
\end{equation}

\textbf{\textit{Recovery of the integral kernel $\mathbf{w}$} }
In this part, we assume that all the advection coefficients are known and are assumed to be non-zero. Following a similar approach to the recovery process of $\pmb{\nu}$, we select $\mathbf{f}_1$ such that $\mathbf{u}^{(I)}(x,t)=(0,\dots,e^{\lambda_i t}U_i(x),\dots,0)$ $(i=1,\dots,N)$, which is non-trivial only at the $i$-th position, in order to recover $w_{ii}$. The coefficient $\lambda_i$ is a known constant, and $U_i(x)=e^{i \xi \cdot (x)}, \vert \xi \vert^2= \frac{\lambda_i}{d_i}$ satisfies Lemma \ref{formu_interp} for the equation \eqref{1storder}. Consequently, the term $\sum\limits^{N}_{j=1}\nu_{ij} \nabla(w^{l}_{ij} \ast u^{(I)}_{j})$ only retains $\nu_{ii} \nabla(w^{l}_{ii}\ast u^{(I)}_{i})$, while the remaining terms are zero since $u^{(I)}_{j}(x,t)=0$ $(j=\{1,\dots,N\}\backslash \{i\})$. Thus, \eqref{var2w_general} can be transformed into
\begin{equation}\label{var2w_wkk}
    \begin{cases}
    \partial_{t} u^{l(II)}_{i}= d_{i}\Delta u^{l(II)}_{i}-\nabla \cdot [ \nu_{ii} u^{(I)}_{i} \nabla  (w^{l}_{ii}\ast u^{(I)}_{i})], & \quad \text{in}\  Q,\\ 
    u^{l(II)}_{i}(x,0)= 2f^{l}_{i,2}(x), & \quad  \text{in} \ \mathbb{T}^{n}. 
  \end{cases}
\end{equation}

For simplicity, we show the case for $i=1$. Let $\bar{\mathbf{u}}^{(II)}=\mathbf{u}^{1(II)}-\mathbf{u}^{2(II)},$ $\bar{w}_{11}=w^{1}_{11}-w^{2}_{11}.$ Applying the property of convolution, we have 
the following equation for $\bar{u}^{(II)}_1$:
\begin{equation}\label{var2w_barwkk}
    \begin{cases}
    \partial_{t} \bar{u}^{(II)}_1= d_1\Delta \bar{u}^{(II)}_1-\nabla \cdot [ \nu_{11} u^{(I)}_1 (\bar{w}_{11}\ast \nabla u^{(I)}_1) ], & \quad \text{in}\  Q,\\ 
    \bar{u}^{(II)}_1(x,0)= 0, & \quad  \text{in} \ \mathbb{T}^{n}. 
  \end{cases}
\end{equation}

Considering $\omega_1$ as the solution of $-\partial_{t} \omega-d_1 \Delta \omega=0$ in $Q$, if we multiply it on both sides of \eqref{var2w_barwkk} and perform integration by parts, we obtain
\begin{equation}\label{omegai_forwkk}
    \int_{Q} -\nabla \cdot [\nu_{11} u^{(I)}_1 (\bar{w}_{11}\ast \nabla u^{(I)}_1)]\omega_1 dxdt=0.
\end{equation}

By substituting the CGO solution $\omega_1=e^{-|\xi_1|^2 t- \frac{i}{\sqrt{d_1}}\xi_1 \cdot x}$ and the specific expression for $U_1(x)=e^{i \xi \cdot x}$, where $\vert \xi \vert^2=\frac{\lambda_1}{d_1}$, into \eqref{omegai_forwkk} and then separating variables, we obtain
\begin{equation}\label{omega1separable_wkk}
    \int_{\mathbb{T}^n} \int_{\mathbb{T}^n}  \nabla \cdot  [\nu_{11} e^{i \xi \cdot x} (\bar{w}_{11}(y) i\xi e^{i \xi \cdot (x-y)})] e^{- \frac{i}{\sqrt{d_1}}\xi_1 \cdot x} dy dx=0.
\end{equation}

After simplification as we did for the proof for $\nu_{11}$, \eqref{omega1separable_wkk} is equal to:
\begin{equation}\label{decide_wkk}
    \frac{2\lambda_1}{d_1} \nu_{11} \int_{\mathbb{T}^n} e^{i \xi \cdot (-y)}\bar{w}_{11}(y)dy =0.
\end{equation}

As we know $\lambda_1$ and $\mu_{11}$ are non-zero constants, \eqref{decide_wkk} indicates that $\bar{w}_{11}=0$, to another words, we have
\begin{equation}\label{unique_w11}
    w^{1}_{11}= w^{2}_{11}.
\end{equation}

With this approach, we can recover $w_{11},w_{22},\dots, w_{NN}$. Next, we modify $\mathbf{f}_1$ such that it becomes non-trivial at both the $i$-th and $j$-th positions, with the corresponding $\mathbf{u}^{(I)}$ being $e^{\lambda_i t}U_i(x)$ and $e^{\lambda_j t}U_j(x)$, in order to recover $w_{ij}$. Then, the term $\sum\limits^{N}_{j=1}\nu_{ij} \nabla(w^{l}_{ij} \ast u^{(I)}_i)$ now becomes $\nu_{ii} \nabla(w^{l}_{ii} \ast u^{(I)}_i)+\nu_{ij} \nabla(w^{l}_{ij} \ast u^{(I)}_j)$. And \eqref{var2w_general} becomes
\begin{equation}\label{var2w_wkm}
    \begin{cases}
    \partial_{t} u^{l(II)}_i= d_{i}\Delta u^{l(II)}_i-\nabla \cdot [ \nu_{ii} u^{(I)}_i \nabla  (w^{l}_{ii}\ast u^{(I)}_i)+\nu_{ij} u^{(I)}_j \nabla  (w^{l}_{ij}\ast u^{(I)}_j)], & \quad \text{in}\  Q,\\ 
    u^{l(II)}_i(x,0)= 2f^{l}_{i,2}(x), & \quad  \text{in} \ \mathbb{T}^{n}. 
  \end{cases}
\end{equation}

Notice that we have confirmed $\bar{w}_{ii}=0$, thus the equation for $\bar{u}^{(II)}_i$ is:
\begin{equation}\label{var2w_barwkm}
    \begin{cases}
    \partial_{t} \bar{u}^{(II)}_i= d_i\Delta \bar{u}^{(II)}_i-\nabla \cdot [ \nu_{ij} u^{(I)}_j (\bar{w}_{ij}\ast \nabla u^{(I)}_j) ], & \quad \text{in}\  Q,\\ 
    \bar{u}^{(II)}_i(x,0)= 0, & \quad  \text{in} \ \mathbb{T}^{n},
  \end{cases}
\end{equation} where $\bar{w}_{ij}$ denotes $w_{ij}^1-w_{ij}^2$.

For simplicity, we take $i=1$. 
By multiplying both sides of \eqref{var2w_barwkm} by the solution $\omega_1$ of $(-\partial_t-d_1\Delta)\omega_1=0$ and applying integration by parts, we have 
\begin{equation}\label{omegai_forwkm}
    \int_{Q} -\nabla \cdot [\nu_{1j} u^{(I)}_{2} (\bar{w}_{1j}\ast \nabla u^{(I)}_{2})]\omega_1 dxdt=0.
\end{equation}

Once again, we substitute the CGO-form of $\omega_1$ and specific expression for $U_j(x)=e^{i \xi \cdot x}$, where $\vert \xi \vert^2=\frac{\lambda_j}{d_j}$, into \eqref{omegai_forwkm}. This allows us to transform \eqref{omegai_forwkm} into
\begin{equation}\label{omega1separable_wkm}
    \int_{\mathbb{T}^n} \int_{\mathbb{T}^n}  \nabla \cdot  [\nu_{1j} e^{i \xi \cdot x} (\bar{w}_{1j}(y) i \xi e^{i \xi \cdot (x-y)})]  e^{- \frac{i}{\sqrt{d_1}}\xi_1 \cdot x} dydx=0.
\end{equation}

Separating variables, the problem now turns to
\begin{equation}\label{decide_wkm}
    \frac{2\lambda_j}{d_j} \nu_{1j} \int_{\mathbb{T}^n} e^{i \xi \cdot (-y)}\bar{w}_{1j}(y)dy =0.
\end{equation}

Since $\lambda_j$ and $\nu_{1j}$ are non-zero constants, \eqref{decide_wkm} tells that $\bar{w}_{1j}=0$, to another words, we have
\begin{equation}\label{unique_w1j}
    w^{1}_{1j}= w^{2}_{1j} \quad, j=\{1,\dots,N\}\backslash \{1\}.
\end{equation}

By sequentially considering $i=1,\dots,N$ and $j=1,\dots,N$, we can fully recover $\mathbf{w}$ and reach the conclusion that $\mathbf{w}^1=\mathbf{w}^2$. This completes the recovery process. \qed

\subsection{Application to the case of Corollary \ref{appthm_case2}}

Similarly, we can apply the conclusion from the Theorem \ref{mainthm_DR} and \ref{mainthm_Inte2} to prove Corollary \ref{appthm_case2}.

\begin{proof}

Based on Lemma \ref{wellposed_eco2}, the system \eqref{appthm_eq} possesses a unique, global, non-negative solution. Consequently, for any initial function satisfying $\mathbf{(I_1)}$, the condition $\mathcal{M}^{+}_{\mathbf{d}^1}(\mathbf{f})=\mathcal{M}^{+}_{\mathbf{d}^2}(\mathbf{f})$ leads to $\mathbf{d}^1=\mathbf{d}^2$.
Then, for any $\mathbf{f}$ satisfying $\mathbf{(I_1)}$, if we have knowledge of the integral kernel $\mathbf{W}$ and $\mathcal{M}^{+}_{\pmb{\nu}^1,\mathbf{W}^1}(\mathbf{f}^1)=\mathcal{M}^{+}_{\pmb{\nu}^2,\mathbf{W}^2}(\mathbf{f}^2)$, we arrive at the conclusion  $\pmb{\nu}^1=\pmb{\nu}^2$. Similarly, if the advection coefficient $\pmb{\nu}$ is given and $\mathcal{M}^{+}_{\pmb{\nu}^1,\mathbf{W}^1}(\mathbf{f}^1)=\mathcal{M}^{+}_{\pmb{\nu}^2,\mathbf{W}^2}(\mathbf{f}^2)$ is known, we can conclude $\mathbf{W}^1=\mathbf{W}^2$.

\end{proof}

\bibliographystyle{plain}
\bibliography{refer_sec}

\end{document}